\documentclass[brochure,12pt]{bourbaki}
\usepackage[utf8]{inputenc}
\usepackage{mathrsfs}
\usepackage[francais]{babel}
\input{xypic}
\date{Juin 2012}
\bbkannee{64\`eme ann\'ee, 2011-12} 
\bbknumero{1056}
\title{Les espaces de Berkovich sont modérés}
\subtitle{d'après Ehud Hrushovski et François Loeser}
\author{Antoine DUCROS}
\address{Institut de mathématiques de Jussieu \\
\&~Institut universitaire de France\\
Université Pierre et Marie Curie\\
 Case 247\\
 4, place Jussieu \\
 F-75252 Paris Cedex 05}
\email{ducros@math.jussieu.fr}
\renewcommand{\epsilon}{\varepsilon}
\newcommand{\ahat}[1]{\widehat{{\mathbb A}^1_{#1}}}
\newcommand{\phat}[1]{\widehat{{\mathbb P}^1_{#1}}}
\newcommand{\abs}[1]{\mathopen |#1\mathclose |}
\renewcommand{\geq}{\geqslant}
\renewcommand{\leq}{\leqslant}
\newcommand{\zero}{^\circ}
\newcommand{\zeroo}{^{\circ\circ}}
\renewcommand{\phi}{\varphi}
\begin{document}
\maketitle

\subsection*{Conventions préalables} Il sera question
tout au long de ce texte de {\em valuations}, pour lesquelles
deux notations sont en concurrence : la notation
additive et la notation multiplicative. Nous avons 
opté pour la notation
multiplicative, naturelle en géométrie
de Berkovich. Une valuation sur 
un corps $k$ consiste donc en la donnée
: d'un groupe abélien ordonné $G$ noté
multiplicativement d'élément neutre 1 ; et d'une application 
$\abs .$ de $k$ 
vers $G_0:=G\cup\{0\}$ (où $0$ est 
absorbant, et plus petit que tout élément de $G$)
telle que $\abs 0=0, \abs 1=1$, $\abs {ab}=\abs a \cdot \abs b$
et 
$\abs{a+b}\leq \max (\abs a, \abs b)$
pour tout $(a,b)\in k^2$ ; {\em nous
n'exigeons pas que~$|k|=G_0$.}
L'anneau de la valuation $\abs .$ est 
égal à $\{x\in k, \abs x\leq 1\}$ ; on le notera $k\zero$. 
Son idéal maximal est $k\zeroo:=\{x\in k, \abs x <1\}$ ; 
son corps résiduel $k\zero/k\zeroo$ sera noté $\tilde k$. 
Un {\em corps valué} est un corps muni d'une valuation ; 
un {\em corps ultramétrique} est un corps muni d'une valeur absolue
ultramétrique, c'est-à-dire encore d'une valuation à valeurs réelles. 

\section*{Introduction}

À la fin des années quatre-vingt, Vladimir Berkovich a proposé
une nouvelle approche de la géométrie analytique 
 ultramétrique 
(\cite{berkovich1990}, \cite{berkovich1993} ; {\em cf.} 
aussi l'exposé \cite{ducros2007b} de ce séminaire).
L'un de ses traits
distinctifs est qu'elle fournit des
espaces ayant d'excellentes propriétés {\em topologiques}, 
qui se sont avérées très utiles pour des raisons techniques, 
mais aussi psychologiques dans la mesure où 
elles sont souvent remarquablement conformes
à l'intuition classique : par exemple, 
le disque unité fermé est, dans ce contexte, une
partie compacte
et à bord non vide de la
droite affine. 

La connaissance de ces
propriétés topologiques a
récemment 
progressé de 
manière considérable grâce
à l'article \cite{hrushovski-loeser2010}
de Ehud Hrushovski et François Loeser, 
auquel ce texte est consacré. Avant d'en présenter
les grandes lignes, nous allons 
succinctement 
exposer l'état de l'art antérieur à sa parution.
On peut, grossièrement, classer les propriétés
qui étaient connues jusqu'alors en trois catégories. 

\medskip
1) {\em Celles qui relèvent de la topologie générale}. Les
espaces de Berkovich sont localement compacts,
localement connexes par arcs, et de dimension
topologique finie lorsqu'ils sont compacts. Elles ont 
été établies par Berkovich lui-même
(pour l'essentiel dans~\cite{berkovich1990})
 aux débuts de la
théorie. Leurs preuves reposent sur des 
arguments très généraux, 
tel le théorème de
Tychonoff, ainsi que sur une étude {\em explicite} 
du disque unité.

2) {\em Celles qui relèvent de la modération\footnote{Nous ne
chercherons pas ici à définir précisément le terme «modération». 
Il est à prendre dans une acception assez vague, dans l'esprit 
d'{\em Esquisse d'un programme} : il évoque une 
certaine forme de finitude,
ainsi que l'absence de pathologies.}
homotopique.} Elles sont là encore dues
à Berkovich, qui les a prouvées
dans l'article \cite{berkovich1999}. 
Soit $X$ un espace
analytique compact sur un corps
ultramétrique complet $k$.
Supposons que $X$
admet un modèle formel
polystable $\mathfrak X$ sur 
$k\zero$. La combinatoire des singularités
de la fibre spéciale 
${\mathfrak X}\otimes \tilde k$
de~${\mathfrak X}$ peut être codée par 
un polytope $P$ de dimension
inférieure ou égale à celle de $X$ ; Berkovich démontre, 
en se ramenant par localisation étale
à une situation torique qui se
traite à la main, 
que $P$
est homéomorphe 
à un fermé de $X$,
sur lequel $X$ se rétracte par 
déformation. 

À l'aide des altérations de de Jong, 
il en déduit que si la valeur absolue de $k$
 n'est pas triviale, 
 tout espace
$k$-analytique lisse est localement contractile ; 
il montre
plus généralement que c'est le cas de tout espace
$k$-analytique
{\em localement isomorphe à un domaine strictement
$k$-analytique d'un espace $k$-analytique lisse}. 
Donnons quelques très brèves explications
sur les termes employés, 
en
renvoyant 
le lecteur en quête de définitions détaillées
au texte fondateur \cite{berkovich1993}
de
Berkovich. Les {\em domaines
analytiques} d'un espace analytique $X$
sont des sous-ensembles de $X$ qui ont une
structure {\em canonique} d'espace analytique. 
Parmi eux, on trouve entre autres les ouverts de $X$
et certains compacts 
intéressants mais pas, en général, les fermés
de Zariski de $X$ (qui peuvent admettre
plusieurs structures analytiques, à cause
des phénomènes de nilpotence). Quant à 
l'adverbe «strictement»,
il fait 
référence à une condition sur les paramètres de définition. 
Illustrons  ces notions
par un exemple : pour tout $n$-uplet
${\bf r}=(r_1,\ldots, r_n)$ de réels strictement positifs,
le polydisque fermé
${\mathbb D}_{\bf r}$
de polyrayon $\bf r$ est un domaine
analytique 
compact 
de l'espace affine analytique 
${\mathbb A}^{n,\rm an}_k$,
et ${\mathbb D}_{\bf r}$
est strictement $k$-analytique si et seulement
si les $r_i$ appartiennent à $|k^*|^{\mathbb Q}$.
Mentionnons incidemment que
${\mathbb A}^{n,\rm an}_k$ est lisse, mais 
pas ${\mathbb D}_{\bf r}$ car ce dernier a un bord non vide.

\medskip
3) {\em Celles qui relèvent de la modération
en matière de composantes
connexes}. Il y en a essentiellement deux, l'une portant
sur les «parties semi-algébriques» de l'analytifié d'une variété
algébrique, et l'autre sur certaines «familles réelles» d'espaces
analytiques. 
 
- {\em Modération des ensembles semi-algébriques}. Soit $X$ une variété algébrique sur un corps ultramétrique 
complet $k$. Une partie $V$ de son analytifié $X^{\rm an}$ est 
dite {\em semi-algébrique} si elle peut être définie,
localement pour la topologie
de Zariski sur $X^{\rm an}$, par une combinaison booléenne  finie
d'inégalités de la forme $\abs f \Join\lambda \abs g$,
où $f$ et $g$ sont
des fonctions polynomiales, où $\lambda\in \mathbb{R}_+$, 
et où le symbole $\Join$
appartient à $\{<,>,\leq,\geq\}$. 

Toute partie semi-algébrique de $X^{\rm an}$ a un nombre fini
de composantes connexes, elles-mêmes 
semi-algébriques. 
Ce résultat a été établi
 dans \cite{ducros2003} par l'auteur\footnote{Dans \cite{ducros2003}
 seul le cas des variétés
 affines est considéré, mais il est immédiat
 qu'il entraîne le cas général.}  ; 
la preuve repose sur
différents résultats difficiles 
d'algèbre commutative normée, reliant
les propriétés d'une algèbre affinoïde
à celles de sa réduction, qui sont 
dus
à Grauert et Remmert d'une part
(\cite{grauert-remmert1966}), 
à Bosch d'autre part (\cite{bosch1977}). 

- {\em Modération en famille réelle}. Soit $X$ un espace analytique compact et 
soit $f$ 
une fonction analytique sur $X$. Pour tout $\epsilon\geq 0$, notons
$X_\epsilon$ l'ensemble des $x\in X$ tels
que $|f(x)|\geq \epsilon$. Il existe une partition finie
de $\mathscr P$ de ${\mathbb R}_+$ en intervalles,
possédant la propriété suivante : pour tout $I\in \mathscr P$
et tout couple $(\epsilon, \epsilon')$ d'éléments 
de $I$ tel que $\epsilon \leq \epsilon'$, l'application naturelle
$\pi_0(X_{\epsilon '})\to \pi_0(X_\epsilon)$
est bijective. Ce théorème 
a été prouvé par Jérôme Poineau dans~\cite{poineau2008},
à l'aide de deux résultats 
de désingularisation : le 
«théorème de la fibre
réduite» (forme affaiblie du théorème de réduction semi-stable, 
démontrée en toute dimension par Bosch, L\"utkebohmert 
et Raynaud dans \cite{bosch-lutk-ray1995}), et
un théorème
d'élimination de la ramification sauvage, établi
par Epp dans~\cite{epp1973}. 
Indiquons que
le 
cas particulier où
la fonction~$f$
est inversible avait été traité antérieurement par Abbes et Saito dans
leur travail~\cite{abbessaito2002}
sur la ramification sauvage.

\medskip
Mentionnons pour conclure ce survol
une interprétation conceptuelle de l'espace topologique
sous-jacent à un espace analytique -- ou à tout le moins
de sa cohomologie. 
Soit $k$ un corps ultramétrique complet,
soit $k^a$ une clôture algébrique de $k$, 
et soit $X$ une $k$-variété algébrique. La
cohomologie de l'espace topologique $X^{\rm an}_{\widehat{k^a}}$ 
a alors tendance à s'identifier de manière Galois-équivariante, 
lorsque cela a un sens, à la partie
{\em de poids zéro} de la cohomologie «usuelle» de $X_{k^a}$. 
Pour plus de précisions, le lecteur
intéressé pourra consulter : l'article \cite{berkovich2000} de Berkovich, qui traite le
cas d'un corps local,
d'un corps de type fini sur son sous-corps premier, et de $\mathbb C$
(la valeur absolue dans ces deux derniers cas est {\em triviale}) ; 
 et les articles \cite{berkovich2009} et \cite{nicaise2011}
(le premier de Berkovich, le second de Johannes Nicaise), 
qui traitent le cas du corps ${\mathbb C}((t))$ muni d'une valeur
absolue $t$-adique,
en lien avec les familles à un 
paramètre de variétés complexes.

\medskip
{\em Les résultats de Hrushovski et Loeser.} On 
peut les résumer en disant qu'ils étendent
2) et 3) à 
{\em toutes les situations provenant de la géométrie algébrique
(projective).}
Plus précisément, soit $k$ un corps ultramétrique
complet, soit $X$ une $k$-variété
algébrique {\em quasi-projective}, et soit $V$ une partie
semi-algébrique de $X^{\rm an}$. Hrushovski et Loeser démontrent
entre autres les assertions suivantes. 

\medskip
\begin{itemize}
\item[A)] {\em Modération globale}. Il existe un fermé
$S$ de $V$ homéomorphe à un complexe
simplicial fini de
dimension inférieure ou égale
à celle de $X$, et une rétraction 
par déformation de $V$ sur $S$
telle que tous les points d'une trajectoire donnée
aient même image terminale sur~$S$. 

\item[B)] {\em Modération locale.} L'espace topologique
$V$ est localement contractile. 

\item[C)] {\em Modération en famille algébrique.} Soit $Y$
une $k$-variété algébrique et soit $f$ un morphisme
de $X$ vers $Y$. L'ensemble des types d'homotopie des fibres
de $f^{\rm an}|_V: V\to Y^{\rm an}$ est fini. 

\item[D)] {\em Modération en famille réelle.} 
Soit $g\in \mathscr O_X(X)$. 
Pour tout $\epsilon\geq 0$, notons
$V_\epsilon$ l'ensemble
des $x\in V$ tels que $\abs{g(x)}\geq \epsilon$.
 Il existe une partition finie
de $\mathscr P$ de ${\mathbb R}_+$ en intervalles,
possédant la propriété suivante : pour tout $I\in \mathscr P$
et tout couple $(\epsilon, \epsilon')$ d'éléments 
de $I$ tel que $\epsilon \leq \epsilon'$, l'inclusion 
$V_{\epsilon'}\subset V_\epsilon$ est une équivalence
homotopique. 

\end{itemize}

\medskip
{\em Commentaires}. En ce qui concerne A),
soulignons que même
dans le cas où $X$ est projective et lisse et où $V=X^{\rm an}$, 
cette assertion n'était jusqu'alors connue que lorsque
$X$ admet un modèle polystable, {\em cf.} les résultats globaux
mentionnés en 2). 

En ce qui concerne B), notons que $V$
possède une base de voisinages qui sont
des parties semi-algébriques de $X^{\rm an}$ ; cela
permet de déduire~B) 
de A) et de 
la locale contractibilité
des complexes simpliciaux finis.

Par ailleurs, 
tout domaine analytique de $X^{\rm an}$ 
possède une base d'ouverts qui sont des parties
semi-algébriques de $X^{\rm an}$. 
Il s'ensuit que tout espace analytique
localement isomorphe à un domaine
analytique (de l'analytifié) 
d'une variété algébrique est
localement contractile.
Comme un espace analytique 
lisse est localement isomorphe
à un ouvert d'une variété algébrique lisse, 
on retrouve comme cas particulier les résultats 
locaux mentionnés plus haut en 2), 
de surcroît un peu étendus : il n'y a
plus besoin de supposer 
que la valeur absolue de $k$ est non
triviale, ni que les domaines en jeu sont 
{\em strictement} analytiques. 

\medskip
{\em Les méthodes de Hrushovski et Loeser.} L'intérêt 
de leur travail réside non seulement dans les résultats que nous
venons d'évoquer, mais aussi dans les méthodes 
totalement inédites qu'ils 
emploient 
pour aborder ce type
de questions. Elles reposent entièrement 
sur la {\em théorie des modèles des corps non trivialement valués
algébriquement clos} (dont nous utiliserons
l'acronyme anglophone ACVF),
et plus précisément sur des progrès
récents en la matière
dus à Haskell, Hrushovski
 et Macpherson 
(\cite{haskell-hrush-macph2006},~\cite{haskell-hrush-macph2008}).  
Contrairement
à celles qui ont été utilisées 
dans les preuves évoquées aux points 1), 2) et 3)
plus haut, elles ne font {\em aucun} appel à l'étude de la 
réduction des variétés algébriques ou des espaces
analytiques, ni à travers les théorèmes 
sur la réduction des algèbres affinoïdes, 
ni à travers ceux qui assurent
l'existence de modèles pas trop singuliers 
(réduction semi-stable, altérations de de Jong, 
fibre réduite). 

Dans 
\cite{haskell-hrush-macph2006}, Haskell, Hrushovski 
et Macpherson
introduisent une version
enrichie du langage des corps valués,
dans laquelle la théorie
ACVF «élimine les imaginaires», ce qui veut dire
que
le quotient d'un foncteur 
définissable par une relation d'équivalence
définissable est encore un foncteur
définissable. Nous allons donner un exemple d'un tel quotient
qui jouera un rôle important dans la suite ; on utilise à partir
de maintenant le langage de \cite{haskell-hrush-macph2006}.
Soit $(k,\abs .\colon k^*\to G)$ 
un corps valué, et soit $k^a$
une clôture algébrique de $k$
munie d'un prolongement $\abs .\colon (k^a)^*\to G^{\mathbb Q}$. Soit $\mathsf M$ la catégorie
des modèles $F$ de ACVF 
(c'est-à-dire des corps non trivialement
valués et algébriquement clos), 
munis d'un diagramme 
commutatif $$\diagram F\rto^{\abs .}&\abs F\\k^a\uto\rto^{\abs.}
&G_0^{\mathbb Q}\uto\enddiagram$$ dans lequel les flèches verticales
sont des plongements. Soit $T$ le foncteur qui envoie
$F\in \mathsf M$
sur le groupe $$\left\{\left(\begin{array}{cc}
a&b\\0&a\end{array}\right)\right\}_{a\in F^*, b\in F}.$$

Il est $k$-définissable ; on note
$T'$ le sous-foncteur 
$k$-définissable $F\mapsto T(F)\cap {\rm GL_2}(F\zero)$ de 
$T$. 
Le foncteur 
quotient $T/T'$ est 
$k$-définissable par élimination
des imaginaires dans la variante de ACVF
utilisée. On vérifie
immédiatement que l'application qui envoie la
matrice $$\left(\begin{array}{cc}
a&b\\0&a\end{array}\right)$$ sur la boule
de centre $b$ et de rayon $\abs a$ induit une
bijection, fonctorielle
en $F$, entre $T(F)/T'(F)$ et l'ensemble 
$B(F)$ des boules fermées
de $F$ dont le rayon appartient 
à $\abs{F^*}$. Ainsi, $B$ est $k$-définissable. 
L'ensemble $\overline B(F)$ des
boules fermées de $F$ dont le rayon appartient 
à $\abs F$ (le rayon nul est maintenant
autorisé)
est isomorphe à 
$B(F)\coprod F$, fonctoriellement en $F$. 
Le foncteur $\overline B$
est donc lui aussi $k$-définissable. 

Supposons que $G=\mathbb R$ et que
$k$ est complet, et soit $F$ un corps
ultramétrique appartenant à $\mathsf M$.
On dispose alors d'une flèche  $\overline B(F)\to  {\mathbb A}^{1,\rm an}_k$, 
qui est surjective si 
$F$ est maximalement complet
(cela signifie que toute famille décroissante de boules 
fermées de $F$ a une intersection non vide). Si $F=k$ la flèche
$\overline B(F)\to  {\mathbb A}^{1,\rm an}_F$ est injective, et bijective
si $F$ est maximalement complet.  

\medskip
Ce foncteur $\overline B$ peut s'interpréter
comme un cas particulier
d'une construction très générale de Hrushovski et Loeser, 
qui est au cœur de leur article
et que nous allons maintenant présenter.
On ne suppose plus que $G=\mathbb R$ 
(ni que $k$ est complet). 
Soit $X$ une $k$-variété algébrique ; on identifie
$X$ au foncteur $F\mapsto X(F)$ de
$\mathsf M$ vers $\mathsf {Ens}$.

Soit $U$ un sous-foncteur $(k,G)$-définissable de $X$ ; 
il est donné, localement pour la topologie de Zariski de $X$, 
par une combinaison booléenne finie
d'inégalités de la forme $\abs f \Join\lambda \abs g$,
où $f$ et $g$ sont
des fonctions polynomiales
à coefficients dans $k$, où $\lambda\in G$, 
et où le symbole $\Join$
appartient à $\{<,>,\leq,\geq\}$. 
Si $G={\mathbb R}$ et
 si $k$ est complet, les formules
 qui décrivent $U$ définissent 
 sans ambiguïté une partie
 semi-algébrique $U^{\rm an}$
 de $X^{\rm an}$, et toute partie
 semi-algébrique de $X^{\rm an}$
 est de cette forme. 

Hrushovski et Loeser notent $\widehat U$ 
le foncteur qui envoie un corps $F\in \mathsf M$
sur l'ensemble des {\em types stablement dominés 
$F$-définissables situés sur $U$}. Nous ne donnerons
pas dans cette
introduction la définition de type stablement dominé ; 
c'est l'une des notions centrales
du livre \cite{haskell-hrush-macph2008}, et 
nous en disons quelques mots dans ce texte au paragraphe~\ref{typstabdom}.
La formation de $\widehat U$ est fonctorielle en $U$, 
pour les applications définissables. Le foncteur~$\widehat U$
jouit des propriétés suivantes.

\medskip
a)  Le foncteur $\widehat U$ est muni d'un plongement naturel 
$U\hookrightarrow \widehat U$. 

b) Le foncteur $\widehat U$ est~$(k,G)$-pro-définissable, et~$(k,G)$-définissable
si $X$ est de dimension $\leq 1$. 

c) Pour tout corps $F\in \mathsf M$ et toute fonction polynomiale
$g$ à coefficients dans $F$ sur un ouvert de Zariski $V$ de $X_F$,
la fonction $\abs g : V(F)\cap U(F)\to \abs F$ se prolonge
en une fonction $\abs g: \widehat V (F)\cap \widehat U(F)\to \abs F$. 
On munit $\widehat U(F)$ de la topologie la plus grossière
pour laquelle $\widehat V(F)\cap \widehat U(F)$ est ouvert
et $\abs g$ continue pour tout $(V,g)$ comme
ci-dessus (la topologie sur $\abs F$
est celle de l'ordre). 
Si $F\subset F'$,
la flèche
$\widehat U(F)\to \widehat U(F')$ est une injection ; 
elle est ouverte sur son image, mais pas continue en 
général. 

d)  Pour tout $F\in \mathsf M$, le
plongement canonique $U(F)\hookrightarrow \widehat U(F)$ est
 un homéomorphisme sur son image, laquelle est dense. 
 
e)  Supposons que $X$ soit propre et que $U$ soit défini Zariski-localement
par une combinaison booléenne positive d'inégalités {\em larges}. Pour tout 
$F\in \mathsf M$, l'espace topologique $\widehat U(F)$ est alors
{\em définissablement} compact (cela signifie {\em grosso modo}
que tout
ultrafiltre définissable, en un sens convenable, y converge). 
 
f) Supposons que $G={\mathbb R}$
et que $k$ est complet,
 et soit $F$ un corps ultramétrique appartenant à 
$\mathsf M$. On dispose d'une application continue 
$\widehat U(F)\to U^{\rm an}$, qui est une surjection 
topologiquement propre
si $F$ est maximalement complet. Lorsque $F=k$, 
la flèche 
$\widehat U(F)\to U^{\rm an}$ est injective
(elle a pour image les points de $U^{\rm an}$ qui sont 
définissables, en un sens à préciser), et c'est un homéomorphisme
si $F$ est maximalement complet. 

\medskip
{\em Un exemple : le cas où $U=X={\mathbb A}^1_k$.} Le foncteur
$\widehat U$ est alors égal à $\overline B$. On
a vu plus haut que $\overline B$ est définissable, 
ce qui redonne b) dans ce cas particulier. 
Soit 
$F\in \mathsf M$.
Le plongement mentionné en a) est celui qui envoie un élément 
de $U(F)=F$ sur la boule singleton correspondante. La topologie
mentionnée en c) est telle que si $x\in F$, l'application
qui associe à $r\in \abs F$ la boule fermée de centre 
$x$ et de rayon $r$ induit un homéomorphisme de $\abs F$
sur son image. Quant à l'application mentionnée en~f), 
elle coïncide avec la flèche 
$\overline B(F)\to {\mathbb A}^{1,\rm an}_k$
évoquée précédemment. 

\medskip
L'essentiel du travail de Hrushovski et Loeser est 
consacré à l'étude des foncteurs de la forme $\widehat U$. 
Ils sont relativement proches des espaces de Berkovich
en vertu de f) ; mais ils sont en vertu
de c) plus maniables
du point de vue de la théorie des modèles, notamment
pour ce qui concerne
la définissabilité
et la modération. 
Ils apparaissent plus précisément
comme les 
objets d'une géométrie 
{\em qui code l'ensemble des questions de modération 
et de définissabilité en géométrie de
Berkovich}. L'apport majeur de Hrushovski et Loeser est,
en un sens, la mise au jour de cette géométrie ; 
les preuves des assertions A), B), C) et D) ci-dessus 
constituent le premier témoignage de sa fécondité. 
Celles-ci
sont pour l'essentiel 
d'abord établies 
dans le monde des espaces
chapeautés, avant d'être rapatriées à la toute fin 
de l'article (chapitre 13) dans celui des espaces de Berkovich, 
grâce aux liens étroits entre les deux géométries, et notamment
aux applications considérées en~f). 

\medskip
À titre d'exemple,
nous allons décrire plus avant cette
stratégie concernant l'assertion A). Introduisons un 
peu de vocabulaire. On note 
$\Gamma$ (resp. $\Gamma_0$) le
foncteur $F\mapsto \abs{F^*}$ (resp. $F\mapsto \abs F$)
de $\mathsf M$ dans $\mathsf{Ens}$. 
Si $a$ et $b$ sont deux
éléments de $|k|$ avec $a\leq b$, 
le segment $[a;b]$ sera considéré comme un 
sous-foncteur de $\Gamma_0$ 
envoyant $F$ sur
$\{x\in \abs F, a\leq \abs x\leq b\}$. Un {\em segment généralisé} 
est un foncteur qui est une concaténation finie de
segments, avec
identifications des extrémité et origine de deux segments
consécutifs\footnote{On prendra garde qu'un tel foncteur
n'est pas en général définissablement isomorphe à un segment, 
{\em cf}. la remarque~\ref{segmentgen}.}; un 
segment
généralisé possède lui-même deux extrémités.
Nous appellerons
{\em polytope} un sous-foncteur 
de 
$\Gamma_0^n$ (pour un certain $n$)
qui est $k$-définissable ; cela signifie 
qu'il peut être
décrit par une combinaison booléenne d'inégalités 
de la	 forme $a\prod x_i^{n_i}\Join b\prod x_i^{m_i}$, où
$a$ et $b$ appartiennent à $|k|$, les $n_i$ et $m_i$ à $\mathbb N$,
et $\Join$ à $\{\leq, \geq, <,>\}$. 
Si $P$ est un polytope (resp. un segment
généralisé), 
alors pour tout~$F$ appartenant à $ \mathsf M$,
 l'ensemble $P(F)$ est de façon
naturelle un espace topologique (resp. un espace
topologique définissablement compact) ; il existe une notion
naturelle de dimension d'un polytope. 

\medskip
L'avatar chapeauté de l'assertion A), dont celle-ci est déduite, est le suivant. 
{\em On suppose que $X$ est quasi-projective}. Il existe alors :

$\bullet$ un sous-foncteur $S$ de $\widehat U$ ; 

$\bullet$ un polytope $P$ de dimension inférieure
ou égale à celle de $X$,  et un 
isomorphisme~\mbox{$S\simeq P$} définissable 
{\em sur une extension finie\footnote{La pr\'esence de cette extension finie est in\'evitable, 
{\em cf.} la remarque~\ref{gallois}.}
de $k$ contenue dans $k^a$}, tels que la bijection~\mbox{$S(F)\simeq P(F)$} soit un homéomorphisme pour
tout $F\in \mathsf M$ ; 

$\bullet$ un segment généralisé $I$, d'origine $o$ et d'extrémité $e$ ; 

$\bullet$ une transformation naturelle définissable $h: \widehat U \times I \to \widehat U$ 
qui fixe $S$ point par point, induit l'identité au-dessus de $o$
et une rétraction $\widehat U \to S $ au-dessus de $e$,
satisfait les égalités~$h(h(x,t),e)=h(x,e)$ pour tout~$(x,t)$, 
et est telle que $h(F)$ soit continue pour tout $F\in \mathsf M$. 

\medskip
Disons quelques mots de la démonstration. On fixe un
ensemble fini $\mathscr F$ de fonctions définissables
de $\widehat X$ vers $\Gamma_0$ qui
contient la fonction caractéristique de $\widehat U$. 
Dans ce
qui suit, toutes les rétractions, homotopies, déformations, etc.
seront implicitement supposées définissables. On dira qu'une
homotopie $h(.,.)$ {\em préserve} $\mathscr F$ si $\phi \circ h(t,.)=\phi$
pour tout $\phi\in \mathscr F$ et tout $t$. 

Quitte
à agrandir $X$, on peut la supposer projective.
Le but est de déformer $\widehat U$
sur un polytope. On va en réalité 
déformer {\em $\widehat X$ tout entier}
sur  un polytope, en préservant~$\mathscr F$. Cela 
garantira
la stabilité de $\widehat U$ sous l'homotopie construite ; 
que l'image terminale 
de $\widehat U$ soit un polytope résultera
immédiatement d'un argument de définissabilité. Notons qu'il
est important de traiter le cas d'un ensemble $\mathscr F$ quelconque, 
même si le cas où \mbox{$\mathscr F=\{\chi_{\widehat U}\}$}
semble suffire ici : certaines étapes
de la récurrence qui est au cœur
de la preuve
(et dont nous donnons une
description extrêmement 
succincte ci-dessous) requièrent en effet
de savoir préserver un ensemble fini 
arbitraire de
fonctions définissables.

On traite tout d'abord le cas où $X$ est une courbe. Si 
$X={\mathbb P}^1_k$ la démonstration se fait 
plus ou moins «à la main» : dans 
une carte affine standard on construit, en gros, 
l'homotopie requise
en faisant croître le rayon des boules et en 
stoppant lorsqu'il convient (rappelons que 
$\widehat {{\mathbb A}^1_k}$ est l'espace des boules fermées).
Si $X$ est une courbe quelconque, on choisit un morphisme fini
et plat $f:X\to {\mathbb P}^1_k$, et l'on cherche
à construire
une homotopie
sur $\widehat  {\mathbb P}^1_k$
qui se relève sur $\widehat X$ en une homotopie
répondant aux conditions prescrites.
 C'est possible grâce à l'étude directe 
 de ${\widehat{\mathbb P}^1_k}$ déjà menée, et
grâce à
un lemme assurant
que le cardinal des fibres de $\widehat f$ a un comportement 
raisonnable, en tant que fonction à valeurs entières
sur $\widehat{{\mathbb P}^1_k}$. Plus
précisément, on prouve d'abord que le
«sens de variation» de cette fonction
est {\em génériquement} ce qu'on 
attend, puis un argument de définissabilité permet de
passer de ce fait générique à une assertion ferme. 

\medskip
Pour construire la rétraction dans le cas général, 
on procède par récurrence sur la dimension $n$ 
de $X$ (après s'être ramené au cas équidimensionnel). 
Le caractère projectif de $X$ assure l'existence d'une
modification $X'\to X$, qui est un éclatement de centre
un 
sous-schéma de dimension $0$, telle que $X'$ admette une
fibration propre en courbes sur une variété 
projective $Y$ 
purement de dimension $n-1$. Soit $D$ la réunion 
des diviseurs exceptionnels de $X'\to X$. 
On cherche à construire 
une rétraction 
de $\widehat{X'}$ sur un polytope, qui préserve 
$\mathscr F$ et
stabilise
$\widehat D$.  Cela assurera qu'elle
descend à $\widehat X$, chaque composante connexe
de $D$, et partant de $\widehat D$, s'écrasant sur un point. 
Pour garantir la stabilisation de $\widehat D$, il suffit de rajouter
sa fonction caractéristique à l'ensemble $\mathscr F$ ; 
on peut ainsi oublier~$D$. 

Pour rétracter $\widehat{X'}$ sur un polytope
en préservant $\mathscr F$ 
l'idée est, très grossièrement,
d'utiliser la fibration $\widehat{X'}\to \widehat Y$ 
en combinant une rétraction 
convenable de la base
 (fournie par l'hypothèse de récurrence) et une 
 rétraction fibre à fibre bien choisie
 (fournie par le cas des courbes déjà traité, dans sa variante relative.)
Ce procédé 
permet 
de construire une homotopie
$h_0$ ayant les propriétés requises, 
mais qui n'est définie 
que sur une partie
de $\widehat{X'}$ 
de la forme $(\widehat {X'}\setminus\widehat{\Delta})\cup \widehat {D_0}$, 
où $\Delta$
et $D_0$ sont deux diviseurs sur $X'$
(l'un des points délicats est la vérification
de la continuité de $h_0$).
Pour obtenir une homotopie 
$h$ définie sur $\widehat {X'}$ tout entier, 
il faut encore travailler. On commence par faire 
précéder
$h_0$ d'une homotopie $h_1$ dite {\em d'inflation} qui
fixe $\widehat {D_0}$ point par point,
préserve $\mathscr F$ et
éloigne un peu 
de $\widehat \Delta$
les points de $\widehat \Delta-\widehat {D_0}$.
On obtient ainsi une homotopie 
qui est définie sur $\widehat{X'}$
tout entier, préserve
$\mathscr F$ et déforme $\widehat{X'}$
sur un polytope,
mais ce dernier n'est plus nécessairement fixé
point par point pendant toute la durée du processus, 
l'homotopie $h_1$ pouvant le perturber. 
On résout le
problème en faisant suivre la
concaténation de $h_1$ et $h_0$ 
d'une troisième homotopie $h_2$, dont 
l'essentiel
de la construction 
se déroule dans le monde polytopal. 

\medskip
Dans les grandes
lignes que nous venons d'esquisser, la preuve semble
reposer essentiellement sur des idées géométriques. 
Mais les arguments de théorie des modèles y sont 
omniprésents, notamment ceux tournant autour de la 
définissabilité et de la compacité
(au sens modèle-théorique du terme). Notons
qu'à chaque fois que l'on veut
appliquer cette dernière,
il est nécessaire de considérer {\em tous}
les modèles de ACVF contenant le corps $k$. C'est 
pourquoi il est indispensable de faire intervenir
tout au long de la preuve des corps valués {\em quelconques},
même si la motivation initiale porte sur les corps
ultramétriques
complets ; pour 
plus de commentaires sur ce sujet, {\em cf.} \ref{digression} {\em infra}.

\medskip
Pour conclure, mentionnons qu'Amaury Thuillier travaille
actuellement à l'extension des résultats de modération homotopique
de Berkovich (aux espaces non nécessairement lisses et/ou n'admettant
pas nécessairement de modèle polystable) par des méthodes 
plus proches de la stratégie originale de Berkovich : utilisation des 
altérations de de Jong et descente homotopique -- c'est la mise en
œuvre de ce dernier point
qui nécessite de nouvelles idées. 

\medskip
\subsection*{Remerciements}

Je tiens à faire part de toute ma gratitude à 
Ehud Hrushvoski et François Loeser pour toutes les discussions
que j'ai eues avec eux depuis maintenant trois ans au sujet de leur travail. Je suis également extrêmement
redevable à tous les participants du groupe de travail qui s'est tenu sur le sujet de 2009 à 2011 à l'{\em Institut
de mathématiques de Jussieu} : Luc B\'elair,
Elisabeth Bouscaren, Zoé Chatzidakis, Françoise Delon, Deirdre Haskell, 
Pierre Simon, ainsi que Martin Hils ; je sais
particulièrement gré à ce dernier
d'avoir relu très attentivement une première version 
de ce texte et fait un grand nombre
de suggestions qui m'ont permis de l'améliorer
significativement. 

\medskip
Je bénéficiais par ailleurs, lors de la rédaction de ce rapport,
du soutien de l'ANR à travers le projet {\em Espaces de Berkovich} (07-JCJC-0004-CSD5). 

\section{Un peu de théorie des modèles}

\subsection{Langages, formules, énoncés}

{\em Références : } le lecteur intéressé pourra
par exemple consulter~\cite{marker2002}
ou~\cite{tent-ziegler2012}. 

\begin{defi} Un {\em langage} est une collection de symboles comprenant : 

$\bullet$ les symboles logiques usuels, ainsi qu'une liste (dénombrable)
de variables ; 

$\bullet$ des symboles de {\em fonctions}, chacun ayant une arité fixée ; 

$\bullet$ des symboles de {\em relations}, chacun ayant une arité fixée. 

\end{defi}

En général, pour décrire un langage, on se contente de donner la liste
des symboles de fonctions et de relations, les autres étant plus ou moins 
implicites ; les fonctions d'arité $0$ sont le plus souvent
appelées {\em constantes}. Dans ce qui suit,
les langages seront également implicitement
supposés contenir un symbole de relation $=$ d'arité 2. 

Donnons quelques exemples. Le langage $\mathscr L_{\rm ann}$
des anneaux comprend
deux constantes~$0$ et $1$ et deux symboles fonctionnels d'arité 2, à savoir
$+$ et $\times$. Le langage $\mathscr L_{\rm co}$
des corps ordonnés comprend les mêmes symboles,
ainsi qu'un symbole de relation $\leq$ d'arité 2. Le langage 
$\mathscr L_{\rm gao}$ des groupes
abéliens ordonnés (notés multiplicativement, avec adjonction formelle d'un plus petit
élément absorbant) comprend
une constante 1, une constante~0, un symbole de fonction $\times$ d'arité 2,
 et un symbole de
relation $\leq$ d'arité 2. 

Une {\em structure} d'un langage $\mathscr L$ est un ensemble 
non vide $M$
muni : pour chaque symbole fonctionnel $f$
d'arité $n$ de $\mathscr L$, 
d'une fonction $f_M: M^n\to M$ (souvent notée
simplement $f$ s'il n'y a pas d'ambiguïté) ; 
et pour chaque symbole de relation $\mathscr R$ 
d'arité~$n$,
d'une relation à $n$ termes $\mathscr R_M$ (ou simplement $\mathscr R$), 
c'est-à-dire d'un sous-ensemble de $M^n$. On demande que
la relation $=_M$ soit l'égalité usuelle. 
Ainsi, une structure du langage des corps ordonnés
est un ensemble muni de deux éléments distingués $0$ et $1$, 
de deux lois de composition internes $+$ et $\times$, et d'une relation
binaire $\leq$. Une {\em sous-structure} d'une structure $M$
est une partie non vide de $M$ qui est stable par les
(interprétations des) symboles fonctionnels. 
On se permettra d'\'ecrire $N\subseteq M$ pour signifier que
$N$ est une sous-structure de $M$. 

On peut écrire des {\em formules} dans un langage $\mathscr L$ en agençant des symboles
de $\mathscr L$ selon les règles syntaxiques usuelles, en respectant les arités des symboles
de fonction et de relation. Par exemple, 
$$\forall x, \exists y, (y+x)\leq z \;{\rm ou}\;(yt \geq u+1\;{\rm et}\;x=t)$$ est une 
formule du langage $\mathscr L_{\rm co}$.
Une formule peut contenir des variables {\em libres} (qui ne
sont pas quantifiées) ; dans la formule ci-dessus, 
$z, u$ et $t$ sont libres. Une formule sans variable 
libre est appelée un {\em énoncé}. À titre
d'illustration, 
$$\forall x,\exists y, x+y=0$$ est un énoncé 
de $\mathscr L_{\rm ann}$. 
Si $\mathscr L$
est un langage et si $M$ est une structure de $\mathscr L$, tout énoncé
de $\mathscr L$ admet une interprétation dans $M$, laquelle
peut être vraie ou fausse : par
exemple, l'énoncé $\forall x,\exists y, x+y=0$ est vrai dans 
la structure $\mathbb Z$  de $\mathscr L_{\rm ann}$,
 mais faux
dans sa structure $\mathbb N$. 

Une formule $\phi$ de $\mathscr L$
donne naissance, si l'on
remplace certaines de ses variables libres
par des éléments de $M$, à une {\em formule du langage $\mathscr L$ à paramètres
dans $M$} ; lorsqu'on spécialise ainsi 
{\em toutes}
les variables libres de $\phi$,
on obtient un {\em énoncé du langage $\mathscr L$ à paramètres
dans $M$}. Un tel énoncé possède une valeur de vérité dans
$M$, et plus généralement dans toute structure 
$M'\supseteq M$.

\subsection{Langages multisortes}

Les définitions des langages, structures, énoncés et formules
s'étendent dans un contexte un peu plus général, dit {\em multisorte}. 
On se donne un ensemble $\mathscr S$. 
Un langage 
{\em d'ensemble des sortes égal à $\mathscr S$}
se définit comme un langage au sens précédent, 
à ceci près qu'il ne suffit plus de fixer les arités des symboles 
fonctionnels et relationnels : on doit également préciser 
la ou les {\em sortes}
dans lesquels vivent leurs arguments, 
et leurs valeurs pour les fonctions. Une structure $M$
d'un tel
langage est une famille $(\mathcal S(M))_{\mathcal S\in \mathscr S}$
d'ensembles non vides (on dit que $\mathcal S(M)$ est «la partie
de sorte $\mathcal S$», ou plus simplement «la sorte $\mathcal S$», de la structure $M$), 
munie d'une 
interprétation des symboles par de vraies relations
ou fonctions. Les définitions des énoncés et formules
dans ce cadre sont {\em mutatis mutandis} celles du 
paragraphe précédent. 

\begin{exem} Le {\em langage des corps valués} $\mathscr L_{\rm val}$
est un langage 
à trois sortes : $\mathsf F$, $\mathsf R$ et $\Gamma_0$. 
Ses symboles sont : 

$\bullet$ les constantes $0_{\mathsf F}$ et $1_{\mathsf F}$, 
$0_{\mathsf R}$ et $1_{\mathsf R}$, $0_{\Gamma_0}$ et $1_{\Gamma_0}$ (la sorte est indiquée
en indice) ; 

$\bullet$ les fonctions $+_{\mathsf F}$ et $\times_{\mathsf F}$, à deux arguments dans la sorte
$\mathsf F$ et
à valeurs dans la sorte~$\mathsf F$ ; 

$\bullet$ les fonctions $+_{\mathsf R}$ et $\times_{\mathsf R}$, à deux arguments dans 
la sorte 
$\mathsf R$ et
à valeurs dans la sorte~$\mathsf R$~; 

$\bullet$ la fonction $\times_{\Gamma_0}$,  à deux arguments dans la sorte
$\Gamma_0$ et
à valeurs dans la sorte $\Gamma_0$~; 

$\bullet$ la fonction $\abs .$ à un argument,
de la sorte $\mathsf F$ vers la sorte $\Gamma_0$ ; 

$\bullet$ la fonction ${\mathsf{res}}$ à deux arguments, 
de la sorte $\mathsf F$ vers la sorte $\mathsf R$ ; 

$\bullet$ la relation binaire $\leq$ sur la  sorte $\Gamma_0$.

\end{exem}

Soit $(k,\abs . \colon k \to G_0)$ un corps valué
{\em (rappelons
qu'on ne demande pas que $\abs . \colon k\to G_0$ soit surjective)}.
On peut le voir de fa\c con naturelle comme une 
structure de $\mathscr L_{\rm val}$, dont les sortes sont 
$$\mathsf F(k,\abs . \colon k \to G_0)=k, \;\mathsf R(k,\abs . \colon k \to G_0)=\tilde k,
\;{\rm et}\; \Gamma_0(k,\abs . \colon k \to G_0)=G_0.$$ 
Les symboles $0,1,+$ et $\times$ indexés par
les différentes sortes ont le sens que l'on imagine, de même que
la fonction $\abs.$ et la relation $\leq$. Quant à ${\mathsf{res}}$, 
elle envoie un couple $(x,y)$ sur $\widetilde{(x/y)}$ si $\abs x\leq \abs y$ 
et $y\neq 0$
et sur (disons) $0$ sinon.
\begin{rema}\label{convcorpsval}
Au lieu d'écrire «soit $(k,\abs . : k\to G_0)$ un corps valué» nous dirons
souvent plus simplement «soit $(k,G_0)$ un corps valué», et nous 
identifierons $(k,G_0)$ à la structure de $\mathscr L_{\rm val}$
qu'il définit. Lorsque nous écrirons «soit $k$ un corps valué» {\em sans mentionner explicitement $G_0$},
cela signifiera qu'on s'intéresse
au corps valué $(k,\abs k)$ ; autrement dit, 
on voit $k$ comme une structure de $\mathscr L_{\rm val}$
dont les sortes
sont \mbox{$\mathsf F(k)=k$,} $\mathsf R(k)=\tilde k$, et $\Gamma_0(k)=\abs k$. 
\end{rema}

\subsection{Théories}

\begin{defi} Soit $\mathscr L$ un langage, éventuellement multisorte. Une {\em théorie}
$T$ dans le langage $\mathscr L$ est un ensemble d'énoncés de $\mathscr L$
qui est consistant ; cela signifie qu'il existe une structure $M$ de $\mathscr L$
dans laquelle tous les énoncés de $T$ sont vrais. On dit qu'une telle
structure est un {\em modèle} de $T$. \end{defi}

Si $M$ est une structure de $\mathscr L$, l'ensemble
des énoncés vrais dans $M$ est une théorie, qu'on appelle la théorie de
$M$. Voici maintenant quelques exemples moins triviaux qui sont
très fréquemment considérés. 

\medskip
$\bullet$ {\em Dans le langage ${\mathscr L}_{\rm ann}$}. La théorie des corps, qui 
comprend
les énoncés de ${\mathscr L}_{\rm ann}$
vrais dans n'importe quel 
corps, comme par exemple
$\forall x, (x\neq 0)\Rightarrow(\exists y, xy=1)$. On définit de même 
la théorie des corps algébriquement clos,
la théorie des corps algébriquement clos de caractéristique fixée, etc.

$\bullet$ {\em Dans le langage ${\mathscr L}_{\rm co}$}. La théorie des corps
ordonnés ; la théorie des corps
réels clos, c'est-à-dire des corps ordonnés $R$ tels que tout élément positif
de $R$ soit un carré dans $R$ et tels que tout polynôme de degré impair
à coefficients dans $R$ ait une racine dans~$R$. 

$\bullet$ {\em Dans le langage ${\mathscr L}_{\rm gao}$}. La théorie
des groupes abéliens ordonnés ; la théorie des groupes abéliens
ordonnés divisibles non triviaux, dite DOAG ({\em divisible ordered abelian
groups}). 

$\bullet$ {\em Dans le langage ${\mathscr L}_{\rm val}$}. La théorie
des corps valués, qui comprend  les énoncés vrais dans toute
structure $(k,G_0)$. Et la théorie des corps
non trivialement valués algébriquement clos, dite ACVF 
({\em algebraically closed valued fields}),
qui comprend les énoncés vrais dans tout corps
valué et algébriquement clos $k$ tel que
$|k^*|\neq\{1\}$. 

\begin{rema}\label{modegalmod} Soit
$T$ l'une des théories que
l'on vient d'énumérer. On vérifie
que les modèles de $T$
sont {\em exactement} les structures dont elle porte le nom,
car celles-ci sont caractérisées par un ensemble d'énoncés
dans le langage concerné  ;
par exemple, les modèles de ACVF sont exactement
les corps non trivialement valués et algébriquement clos. 
\end{rema}

\medskip
On dit que $T$ {\em élimine les quantificateurs} si pour toute
formule $\phi(x_1,\ldots, x_n)$ dans le langage $\mathscr L$
à variables libres $x_1,\ldots, x_n$, il existe
une formule $\psi(x_1,\ldots, x_n)$ dans le langage $\mathscr L$ 
à variables libres $x_1,\ldots, x_n$  et {\em sans quantificateurs}
telle que $$\forall (x_1,\ldots, x_n), \;\;\; \phi(x_1,\ldots,x_n)\iff \psi(x_1,\ldots, x_n)$$ soit un énoncé
de $T$. 

\begin{lemm}\label{modcomp} 
Supposons que $T$ élimine les quantificateurs, soit $A$ une structure
du langage $\mathscr L$, et soient $M$ et $M'$ deux modèles de $T$ 
tels que $A\subseteq M$ et $A\subseteq M'$.
Soit $\Phi$ un énoncé du langage $\mathscr L$
à paramètres dans $A$. Il est alors vrai dans $M$ si et seulement
si il est vrai dans $M'$.
\end{lemm}

\begin{proof}
Cela résulte de la définition d'une sous-structure
si $\Phi$ est sans quantificateurs, et l'hypothèse faite sur $T$ permet de se ramener à
ce cas. 
\end{proof}

La théorie des corps algébriquement clos, 
la théorie des corps réels clos, la théorie 
DOAG et la théorie ACVF éliminent les quantificateurs
C'est pour bénéficier
de l'élimination des quantificateurs 
qu'on a choisi, en définissant ACVF, de se limiter aux corps algébriquement clos
{\em non trivialement valués} ; en effet, la théorie $T$
des corps 
valués
algébriquement clos
généraux (avec valuation triviale autorisée) dans le langage $\mathscr L_{\rm val}$
n'élimine pas les quantificateurs. Pour le voir, 
fixons un corps algébriquement clos $k$ trivialement valué, et une 
extension {\em non trivialement valuée}
et algébriquement
close $K$ de $k$. L'énoncé
«~$\exists x, x\neq 0 \;{\rm et}\;\abs x\neq 1$~» est
alors par construction 
faux dans $k$ et vrai dans $K$ ; 
il en résulte en vertu du lemme~\ref{modcomp}
que $T$ n'élimine pas les quantificateurs.

\subsection{Ensembles et foncteurs définissables}\label{Ensfonct} 

Soit $T$ une théorie dans un langage 
$\mathscr L$ dont on note $\mathscr S$ l'ensemble
des sortes ; 
on suppose que 
$T$ élimine les quantificateurs. On note~$\Sigma(\mathscr L)$
la collection de toutes les structures de~$\mathscr L$, à laquelle
on rajoute l'ensemble vide. Soit~$A\in \Sigma(\mathscr L)$. 
On note $\mathsf M_A$ la catégorie des modèles 
de $T$ tels que~$A\subseteq M$ (si~$A=\emptyset$, on a donc
affaire à la catégorie de {\em tous} les modèles de~$T$) ; 
les morphismes sont les inclusions comme sous-structures. 
Pour tout~$\mathcal S\in \mathscr S$, on note~$\mathcal S_A$
la restriction de~$M\mapsto \mathcal S(M)$ à~$\mathsf M_A$. 

\begin{defi}\label{defidefi} Soit $M$ 
appartenant à $\mathsf M_A$ 
et soit $(n_{\mathcal S})_{{\mathcal S}\in \mathscr S}$ une famille d'entiers presque tous nuls. 
Un sous-ensemble
$E$ de $\prod \mathcal S(M)^{n_{\mathcal S}}$ 
(resp. un sous-foncteur~$D$ de~$\prod \mathcal S_A^{n_{\mathcal S}}$)
est dit {\em $A$-définissable} s'il existe 
une formule $\Phi([\underline x_{\mathcal S}]_{\mathcal S})$ du langage $\mathscr L$,
à paramètres
dans $A$, en variables libres $(\underline x_{\mathcal S})$ où chaque
$\underline x_{\mathcal S}$ est un $n_{\mathcal S}$-uplet de variables de la sorte $\mathcal S$, 
telle que $$E=\left\{(\underline a_{\mathcal S})\in \prod \mathcal S(M)^{n_{\mathcal S}}, 
\Phi([\underline a_{\mathcal S}])\right\}$$

(resp. $F(N)=\left\{(\underline a_{\mathcal S})\in \prod \mathcal S(N)^{n_{\mathcal S}}, 
\Phi([\underline a_{\mathcal S}])\right\}$ pour tout~$N\in \mathsf M_A$). 
\end{defi}

Soit 
$E$ un sous-ensemble $A$-définissable de $\prod \mathcal S(M)^{n_{\mathcal S}}$
et 
soit $E'$ un sous-ensemble $A$-définissable de $\prod \mathcal S(M)^{m_{\mathcal S}}$.
Une application $f: E\to E'$ est dite 
$A$-définissable si son graphe est  une partie
$A$-définissable de~$\prod \mathcal S(M)^{n_{\mathcal S}+m_{\mathcal S}}$. 
Si $f$ est $A$-définissable, 
$f(E)$ est un sous-ensemble $A$-définissable
de $\prod \mathcal S(M)^{m_{\mathcal S}}$. 

De même, soit~$D$ un sous-foncteur~$A$-définissable 
de $\prod \mathcal S_A^{n_{\mathcal S}}$
et 
soit $D'$ un sous-foncteur~$A$-définissable de $\prod \mathcal S_A^{m_{\mathcal S}}$.
Une transformation naturelle~$f: D\to D'$ est dite 
$A$-définissable si son graphe est  un sous-foncteur
$A$-définissable de~$\prod \mathcal S_A^{n_{\mathcal S}+m_{\mathcal S}}$. 
Si $f$ est~$A$-définissable, 
$f(D)$ est un sous-foncteur~$A$-définissable
de $\prod \mathcal S_A^{m_{\mathcal S}}$.

\medskip
Soit  $E$ un sous-ensemble~$A$-définissable de
$\prod \mathcal S(M)^{n_{\mathcal S}}$ et soit $\Phi$
comme dans la définition~\ref{defidefi}. 
Le foncteur
$$N\mapsto \left\{(\underline a_{\mathcal S})\in \prod \mathcal S(N)^{n_{\mathcal S}}, \Phi([\underline a_{\mathcal S}])\right\}$$
est un sous-foncteur~$A$-définissable de 
$\prod \mathcal S_A^{n_{\mathcal S}}$. 
{\em Supposons~$A\neq \emptyset$}. Ce foncteur 
ne dépend alors que de~$E$, et pas du choix
de la formule 
$\Phi$ utilisée pour décrire ce dernier. 
En effet, si $\Psi$ est une autre formule 
décrivant $E$, le fait que $\Phi$ et $\Psi$ décrivent le même
ensemble au niveau du modèle $M$ se propage 
en vertu du lemme~\ref{modcomp}~à tout
$N\in \mathsf M_A$ (c'est ici qu'on utilise le fait
que~$A$ est non vide : le lemme~\ref{modcomp}
requiert l'existence d'une sous-structure commune 
aux deux modèles qu'il met en jeu).  
Il est dès lors licite de noter ce foncteur~$\underline E$. 

\begin{rema} L'assertion précédente
est fausse en général lorsque~$A=\emptyset$. 
Par exemple, supposons que~$T$ est la théorie des corps
algébriquement clos, et soit~$D$
(resp.~$D'$) le foncteur qui associe à
un modèle~$F$ de~$T$
l'ensemble~$\{x\in F,x+1=x\}$
(resp. l'ensemble~$\{x\in F, x+x=1\}$). Les foncteurs~$D$ et~$D'$
sont~$\emptyset$-définissables, et si~$F$ est un modèle de~$T$
de caractéristique 2, alors~$D(F)=D'(F)=\emptyset$. Mais~$D$ et~$D'$
ne coïncident pas pour autant : si~$F$ est un modèle de~$T$
de caractéristique
différente de~$2$ alors~$D(F)=\emptyset$ et~$D'(F)=\{1/2\}$. 
\end{rema}

Soit $M\in \mathsf M_A$. Notons $\mathsf D_{A,M}$ la catégorie définie comme suit. 

$\bullet$ Ses objets sont les
couples $((n_{\mathcal S}),E)$ où $(n_{\mathcal S})_{\mathcal S\in \mathscr S}$ est une famille d'entiers presque tous nuls, et $E$ 
un sous-ensemble $A$-définissable de $\prod \mathcal S(M)^{n_{\mathcal S}}$. 

$\bullet$ Ses morphismes sont les applications $A$-définissables. 

On définit de même la catégorie ${\mathbb D}_A$ en remplaçant sous-ensembles $A$-définissables
de $\prod \mathcal S(M)^{n_{\mathcal S}}$ par sous-foncteurs $A$-définissables de
 $\prod \mathcal S_A^{n_{\mathcal S}}$, 
et applications $A$-définissables par transformations naturelles $A$-définissables. 
Si~$A$ est non vide, 
les flèches
$E\mapsto \underline E$ et $D\mapsto D(M)$ établissent alors un {\em isomorphisme} entre
les catégories $\mathsf D_{A,M}$ et ${\mathbb D}_A$. 

\medskip
Donnons maintenant quelques exemples. 

$\bullet$ On suppose que $\mathscr L$ est le langage des corps ordonnés, et que
$T$ est la théorie des corps réels clos. Soit $R_0$ un modèle
de $T$ et soient $a$ et $ b$ deux éléments de $R_0$. Le sous-ensemble
$E:=\{x\in R_0, a\leq x\leq b\}$ de $R_0$ 
est visiblement définissable sur la sous-structure $A:={\mathbb Q}(a,b)$
de $R_0$. 
Le foncteur $\underline E$ envoie
$R\in \mathsf M_A$ sur $\{x\in R, a\leq x\leq b\}$ ; on 
le note $\mathopen[a;b\mathclose]$. 

Le foncteur~$\mathopen[0;1\mathclose] : R\mapsto \{x\in R,0\leq x\leq 1\}$ est défini 
sur~$\mathsf M_\emptyset$ et est~$\emptyset$-définissable. 

$\bullet$ On suppose que $\mathscr L$ est le langage des corps valués, 
et que $T$ est égale à ACVF. Soit $F_0$ un modèle de $T$.

- Soient
$a$ et $b$ deux éléments de $\abs {F_0}$. Le sous-ensemble
$E:=\{x\in \abs {F_0}, a\leq x\leq b\}$ de $\abs{F_0}$
est visiblement définissable sur la sous-structure
$A:=(k_0,\langle \abs{k_0},a,b\rangle)$ de $F_0$,
 où $k_0$ est 
le sous-corps premier de $F_0$. Le foncteur $\underline E$ envoie
un modèle $F$ appartenant à $\mathsf M_A$
sur l'ensemble~$\{x\in \abs F, a\leq x\leq b\}$ ; on 
le note $\mathopen[a;b\mathclose]$. 

Le foncteur~$\mathopen[0;1\mathclose] : F\mapsto \{x\in \abs F,0\leq x\leq 1\}$ est défini 
sur~$\mathsf M_\emptyset$ et est~$\emptyset$-définissable.

- Soit $\lambda \in F$ et soit $r\in \abs F$. Le sous-ensemble 
$E:=\{y\in F, \abs{y-\lambda}\leq r\}$ est visiblement définissable sur la sous-structure
$A:=(k_0(\lambda), \langle \abs{k_0(\lambda)},r\rangle)$ de
$F_0$. 
Le foncteur $\underline E$ envoie
$F\in \mathsf M_A$ sur $\{y\in F, \abs{y-\lambda}\leq r\}$ ; on 
le note $b(\lambda,r)$. On dit que c'est la {\em boule fermée de centre
$\lambda$ et de rayon $r$.} On définit de même le foncteur $b^{\rm ouv}(\lambda,r)$
(boule {\em ouverte} de centre $\lambda$ et de rayon $r$). 

Les foncteurs~$b(0,1) : F\mapsto \{x\in F,\abs x\leq 1\}$ 
et~$b^{\rm ouv}(0,1) : F\mapsto \{x\in F,\abs x<1\}$ sont définis
sur~$\mathsf M_\emptyset$ et~$\emptyset$-définissables.

\medskip
Il arrive fréquemment que l'on dise d'un foncteur covariant
$h : \mathsf M_A\to \mathsf{Ens}$ qu'il est~{\em$A$-définissable}. Cela 
signifie qu'il existe un foncteur $D\in {\mathbb D}_A$ et un isomorphisme
de foncteurs $h\simeq D$. Le problème est que le foncteur $D$ en question 
n'a {\em a priori} pas de raison d'être canonique, si l'on s'en tient à cette définition. 
En réalité, à chaque fois que l'on emploiera cette terminologie, le foncteur $D$ sera bien défini, 
pour une raison simple : on se donnera toujours {\em implicitement} un peu plus
qu'un foncteur covariant de $\mathsf M_A$ vers $\mathsf{Ens}$, et ces données
supplémentaires rigidifieront la situation. Nous allons expliquer ci-après en détail de quoi
il retourne.

\medskip
Soit $D\in {\mathbb D}_A$. Soit~$B\in \Sigma(\mathscr L)$
telle que~$A\subseteq B$. 
Le foncteur $D$ définit par restriction un élément 
de ${\mathbb D}_B$, noté $D_B$, et partant un foncteur contravariant 
${\rm Hom}_B(.,D_B)$ de ${\mathbb D}_B$ dans $\mathsf {Ens}$ ; 
si $A\subseteq B \subseteq B'$, on a une inclusion naturelle
${\rm Hom}_B(\Delta,D_B)\subset {\rm Hom}_{B'}(\Delta_{B'},D_{B'})$
pour tout $\Delta\in D_B$. 
Remarquons par ailleurs que 
pour tout $M\in \mathsf M_A$, on a $D(M)={\rm Hom}_M(\{*\},D_M)$, où 
$\{*\}$ est un singleton arbitrairement choisi dans un produit
$\prod \mathcal S(M)^{n_{\mathcal S}}$ quelconque (notons
que~$\{*\}$ est un objet 
de $\mathsf D_{M,M}\simeq{\mathbb D}_M$). 

Soit $\mathsf C_A$ la catégorie définie comme suit. 

$\bullet$ Ses objets sont
les couples $(B,D)$, où $B\in \Sigma(\mathscr L)$, où~$A\subset B$, 
et où $D\in {\mathbb D}_B$. 

$\bullet$ L'ensemble des morphismes de $(B,D)$ dans $(B',D')$ est l'ensemble
${\rm Hom}_B(D,D'_B)$ si $B'\subseteq B$, et est vide sinon. 

Si $D\in {\mathbb D}_A$ alors $(B,\Delta)\mapsto {\rm Hom}_B(\Delta,D_B)$
est de façon naturelle
un foncteur contravariant de $\mathsf C_A$ vers $\mathsf{Ens}$, 
noté ${\rm Hom}_{\bullet}(.,D_\bullet)$. 
On dira qu'un foncteur 
contravariant \mbox{$h: \mathsf C_A\to \mathsf {Ens}$} est 
{\em $A$-définissable} s'il existe $D\in {\mathbb D}_A$ tel que 
$h$ soit isomorphe à ${\rm Hom}_{\bullet}(.,D_\bullet)$. 
Le lemme de Yoneda assure que $D$ est dans ce cas uniquement déterminé.

\medskip
On peut maintenant préciser quelles
sont les données implicites évoquées plus haut. Lorsqu'on dira qu'un foncteur covariant 
$h : \mathsf M_A\to \mathsf{Ens}$
est $A$-définissable, il sera sous-entendu : 

- que le foncteur $h$ «se met en familles de façon naturelle», 
c'est-à-dire qu'il 
existe un foncteur contravariant 
$h^\sharp$ de $\mathsf C_A$ dans $\mathsf{Ens}$, 
dont la définition est censée découler 
clairement du contexte, 
 et des 
identifications
naturelles $h(M)=h^\sharp(M,\{*\})$ pour tout $M$
appartenant à $\mathsf M_A$ ;

- que le foncteur $h^\sharp$ est $A$-définissable.

En pratique, la première de ces conditions se produit lorsqu'on s'est donné
une classe d'objets $\mathscr C$, 
lorsque $h(M)$ 
se décrit comme l'ensemble des objets de $\mathscr C$
définis sur $M$, et lorsque la notion de famille $B$-définissable d'objets
de $\mathscr C$ tombe peu ou prou sous le sens : on définit 
alors justement $h^\sharp(B,D)$ comme l'ensemble des familles
$B$-définissables d'objets de $\mathscr C$ paramétrées par 
$D$. 

\medskip
Soit $h$ un foncteur $A$-définissable, c'est-à-dire
un foncteur covariant de $\mathsf M_A$
dans $\mathsf{Ens}$ qui est $A$-définissable
au sens ci-dessus. Il existe un objet
$D\in \mathsf D_A$, canoniquement déterminé, tel 
que $h\simeq D$.
Il en résulte une définition naturelle de sous-ensemble
\mbox{$A$-définissable} de $h(M)$ pour tout $M\in \mathsf M_A$, de
transformation naturelle
$A$-définissable de $h$ vers un autre foncteur $A$-définissable $h'$, 
etc. Si~$A\neq \emptyset$ alors
pour tout $M\in \mathsf M_A$, tout sous-ensemble $A$-définissable $E$
de $h(M)$ induit
un sous-foncteur $A$-définissable $\underline E$ de $h$ ; et
les flèches $E\mapsto \underline E$ et $g\mapsto g(M)$ 
mettent en bijection l'ensemble des parties
$A$-définissables de $h(M)$ et celui des sous-foncteurs
$A$-définissables de $h$. 

Si $B\in \Sigma(\mathscr L)$ est telle que $A\subseteq B$, tout
foncteur $A$-définissable $h$ induit par restriction à $\mathsf M_B$ un 
foncteur $B$-définissable $h_B$. Donnons maintenant
quelques exemples de foncteurs définissables.

\medskip
$\bullet$ On se place dans la théorie des corps
algébriquement clos. Soit $k$ un corps 
et soit $X$ un $k$-schéma de type fini ; il induit 
un foncteur $k$-définissable, encore noté $X$. 

$\bullet$ On se place dans la théorie des corps réels clos. Soit 
$k$ un corps ordonné, et soit $X$ un $k$-schéma de type fini ; il 
induit 
un foncteur $k$-définissable, encore noté $X$. 
Un sous-foncteur $U$ de $X$ est~$k$-définissable
si et seulement si peut être défini Zariski-localement
sur $X$ par une combinaison booléenne d'inégalités entre fonctions
régulières. 

$\bullet$ On se place dans la théorie ACVF. Soit 
$(k,G_0)$ un corps valué, et soit $X$ un $k$-schéma de type fini ; 
il induit
un foncteur $k$-définissable, encore noté $X$. 
Un sous-foncteur $U$ de $X_{(k,G_0)}$ est 
$(k,G_0)$-définissable
si et seulement si il peut être
défini Zariski-localement
sur $X$ par une combinaison booléenne d'inégalités 
de la forme $\abs f\Join \lambda \abs g$, où $f$ et $g$
sont des fonctions régulières, où $\lambda\in G_0$ et
où $\Join\in\{<,>,\leq,\geq\}$.

$\bullet$ On se place dans la théorie ACVF. 
Soit $(k, G_0)$ un corps valué, et soient $(a_i)_{1\leq i\leq n}$ 
et $(b_i)_{1\leq i\leq n}$ deux familles finies d'éléments de $G_0$.
Le foncteur qui envoie $F\in \mathsf M_{(k,G_0)}$ 
sur $$\left(\coprod_i \{x\in \abs F, \min(a_i,b_i)\leq x\leq \max(a_i,b_i)\}\right)/\mathscr R$$ où $\mathscr R$ 
identifie pour tout $i\leq n-1$ l'élément $b_i$ du $i$-ème sommande avec l'élément 
$a_{i+1}$ du $(i+1)$-ème sommande, est $(k, G_0)$ définissable. 
On appelle
un tel foncteur un {\em segment généralisé} ; le foncteur singleton $F\mapsto \{a_1\}$ (resp. $F\mapsto \{b_n\}$)
est appelé son {\em origine} (resp. {\em extrémité}). 

En s'autorisant à identifier origines et/ou extrémités de plus de deux segments, on définit
de même des foncteurs $(k,G_0)$-définissables appelés {\em graphes finis}, dont
certains sont des {\em arbres}. Nous laissons au lecteur le soin d'en donner
des définitions précises. 

\begin{rema}\label{stablyemb}
Soit $D$ un foncteur $G_0$-définissable 
dans la théorie DOAG. Le foncteur
$D\circ \Gamma_0 : F\mapsto D(\abs F)$ 
est alors un foncteur $(k,G_0)$-définissable
dans la théorie ACVF. 
On vérifie que l'application $\Delta\mapsto \Delta\circ \Gamma_0$
établit une {\em bijection} entre l'ensemble
des sous-foncteurs $G_0$-définissables de $D$
et celui des sous-foncteurs $(k,G_0)$ définissables
de $D\circ \abs .$~. En bref, la trace de la théorie ACVF
sur la sorte $\Gamma_0$ coïncide avec la théorie
DOAG. 
De même, la trace de la théorie ACVF
sur la sorte «corps résiduel» coïncide avec la théorie
des corps algébriquement clos. 

\end{rema}

\begin{rema}\label{segmentgen} On prendra garde qu'un 
segment généralisé n'est pas, en général, 
définissablement isomorphe à un segment
$F\mapsto \{x\in \abs F, a\leq x\leq b\}$. 
Le lecteur vérifiera par exemple
que la concaténation~$[0;1]\coprod\limits_{1=0}[0;1]$
(qui est~$\emptyset$-définissable) ne
peut pas être identifiée à un segment, car le langage autorisé
pour une telle identification est $\mathscr L_{\rm gao}$
({\em cf.} la remarque~\ref{stablyemb} ci-dessus), qui
comprend simplement 
la relation d'ordre et la multiplication. Par contre, une concaténation
de segments {\em à extrémités non nulles}
 est encore
(définissablement isomorphe à) un segment.
\end{rema}

\subsection{Élimination des imaginaires ; le langage 
des corps valués de Haskell, Hrushovski et Macpherson}

{\em Références : } nous renvoyons le lecteur à l'ouvrage \cite{haskell-hrush-macph2008}, ainsi
qu'aux transparents d'exposé d'Ehud Hrushovski (\cite{notes-hrush2011}). 

On conserve les notations de la section précédente. Soit $D$
un foncteur $A$-définissable. 
On dira qu'un sous-foncteur $R$ de $D\times D$ est une {\em relation d'équivalence
$A$-définissable} si $R$ est $A$-définissable
et si $R(M)$ est pour tout $M\in \mathsf M_A$
le graphe d'une relation d'équivalence sur $D(M)$. Le foncteur quotient 
$D/R$ est alors bien défini. 

\begin{defi}\label{elimimag} On dit que la théorie $T$ {\em élimine les imaginaires}
dans le langage $\mathscr L$ si pour toute~$A\in \Sigma(\mathscr L)$, pour tout foncteur
$A$-définissable $D$ et toute relation d'équivalence $A$-définissable
$R\subset D\times D$, le quotient $F/R$ est $A$-définissable
(la flèche quotient $F\to F/R$ est alors automatiquement
$A$-définissable). 

\end{defi}

\begin{rema} Pour que~$T$ élimine les imaginaires, il suffit
que la condition ci-dessus soit vérifiée lorsque~$A=\emptyset$ :
cela provient du fait qu'en général, toute relation d'équivalence~$A$-définissable peut être vue, 
en «faisant varier les paramètres», comme une spécialisation d'une
relation~$\emptyset$-définissable. 

\end{rema}

La théorie des corps algébriquement clos
ou celle des corps réels clos éliminent les imaginaires, 
mais ce n'est pas le cas de ACVF : on peut montrer
que le foncteur quotient
$F\mapsto {\rm GL}_2(F)/{\rm GL}_2(F\zero)$  (autrement dit,
le foncteur «espace des réseaux du plan»), défini sur 
$\mathsf M_\emptyset$, n'est pas $\emptyset$-définissable. Notons par
contre que $F\mapsto {\rm GL}_1(F)/{\rm GL}_1(F\zero)$ est isomorphe
à $F\mapsto |F^*|$, et est donc $\emptyset$-définissable. 

Il existe un procédé standard pour enrichir le langage $\mathscr L$
de façon à {\em forcer} la théorie $T$ à éliminer les imaginaires
: il consiste 
à rajouter, 
pour toute famille $(n_{\mathcal S})$ d'entiers presque tous nuls, tout sous-foncteur 
$\emptyset$-définissable $D$ de $\prod \mathcal S_\emptyset^{n_S}$ et toute relation
 d'\'equivalence~$\emptyset$-définissable $R\subset D\times D$ d\'efinissable sans param\`etres, une sorte
«quotient de $D$ par $R$» et un symbole fonctionnel «application quotient
de $D$ vers $D/R$». Ce langage étendu~$\mathscr L^*$ possède les propriétés suivantes : 

$\bullet$ si~$A\in \Sigma(\mathscr L)$, si~$M\in \mathsf M_A$
et si~$D$ est un foncteur~$A$-définissable,
un sous-foncteur de~$D$ (resp. un sous-ensemble
de~$D(M)$) est~$\mathsf A$-définissable au sens de~$\mathscr L^*$
si et seulement si il l'est au sens de~$\mathscr L$ ; 

$\bullet$ le lemme~\ref{modcomp} est encore valable pour~$\mathscr L^*$ ;

$\bullet$ la théorie~$T$ élimine les imaginaires dans le langage~$\mathscr L^*$. 

\medskip
Par contre, la théorie~$T$ n'élimine plus nécessairement
les quantificateurs dans le langage~$\mathscr L^*$. On peut y remédier 
sans modifier l'ensemble des sortes, 
ni la notion 
de foncteur ou de sous-ensemble~$A$-définissable : il suffit de rajouter, 
au langage~$\mathscr L^*$, pour
toute formule (quantifiée)~$\Phi$ de~$L^*$
en les variables libres~$(x_1,\ldots,x_n)$,
un symbole de relation~$n$-aire~$\sigma_\Phi$,
que l'on interprète dans un modèle~$M$ de~$T$ en disant
que~$\sigma_\Phi(m_1,\ldots,m_n)\iff \Phi(m_1,\ldots, m_n)$. 

\medskip
Si le langage ainsi obtenu a un intérêt théorique, il n'est en général 
pas forcément 
très explicite ni très maniable, étant donnée la profusion
de sortes supplémentaires introduites. Aussi cherche-t-on, lorsque c'est possible, 
à mettre en évidence un ensemble limité de quotients tel que l'adjonction
des sortes et symboles fonctionnels correspondants suffise à rendre définissables
{\em tous} les quotients.  
C'est ce que Haskell, Hrushovski et Macpherson ont
fait dans \cite{haskell-hrush-macph2006} à propos de la théorie ACVF. 
Ils ont démontré
que pour qu'elle élimine les imaginaires, il suffit de rajouter
au langage $\mathscr L_{\rm val}$ 
 les sortes et
symboles fonctionnels suivants, où ${\rm GL}_n$ (resp. ${\rm GL}_n\zero$, resp. ${\rm GL}_n\zeroo$) 
désigne
le foncteur qui envoie
un modèle $F$ de ACVF sur ${\rm GL}_n(F)$ (resp. ${\rm GL}_n(F\zero)$, resp. 
$I_n+F\zeroo {\rm M}_n(F\zero)$~) : 

$\bullet$ pour tout $n>0$, une sorte $\mathcal S_n$ 
pour le quotient ${\rm GL}_n/{\rm GL}_n\zero$,
et un symbole fonctionnel $\rho_n$ codant l'application quotient correspondante ; 

$\bullet$ pour tout entier $n>0$, une sorte $\mathcal T_n$ pour le quotient 
${\rm GL}_n/{\rm GL}_n\zeroo$, 
et un symbole fonctionnel $\tau_n$  codant l'application quotient 
$ {\rm GL}_n/{\rm GL}_n\zeroo\to {\rm GL}_n/{\rm GL}_n\zero$. 

\medskip
\begin{rema}Soit $F$ un modèle de ACVF. On peut identifier
naturellement $\mathcal S_n(F)$ à l'espace des réseaux de $F^n$, 
et $\mathcal T_n(F)$, {\em via} l'application
$\tau_n$,  
à l'espace total du «fibré des bases»
d'un certain « fibré vectoriel»~sur $\mathcal S_n(F)$ 
; plus précisément, la fibre de $\tau_n$
en un point correspondant à un réseau $\Lambda$ est l'espace
des bases du $\tilde \Lambda$-espace vectoriel $\Lambda/F\zeroo \Lambda$. 
\end{rema}

Appelons $\mathscr L^*_{\rm val}$ le langage
des corps valués enrichi par les sortes $\mathcal S_n$ et $\mathcal T_n$
et les symboles $\rho_n$ et $\tau_n$,
et convenablement \'etendu de fa\c con \`a assurer 
l'\'elimination des quantificateurs\footnote{Nous avons expliqué
ci-dessus comment atteindre ce dernier objectif par adjonction
d'un ensemble de symboles relationnels. Celui-ci
peut sembler
absolument gigantesque, et peu tangible.
Mais dans~{\em loc. cit.}, Haskell,
Hrushovski et Macpherson exhibent un ensemble
{\em explicite}
et de taille raisonnable de tels symboles qui s'avère
suffisant.}. C'est désormais avec lui
que nous travaillerons, et l'acronyme ACVF désignera 
à partir de maintenant la théorie des corps non trivialement
valués et algébriquement clos 
{\em dans le langage $\mathscr L^*_{\rm val}$}. Cette 
nouvelle version de ACVF élimine par construction
les quantificateurs
et les imaginaires. 
Le {\em foncteur 
des boules fermées} $\overline B$
de $\mathsf M_\emptyset$ vers
$\mathsf {Ens}$ qui associe à un corps $F\in \mathsf M_\emptyset$ 
l'ensemble $\{b(\lambda,r)_F\}_{\lambda\in F,r\in \abs F}$ 
est $\emptyset$-définissable ; en effet, il se décrit simplement,
comme on l'a vu dans l'introduction\footnote{Dans l'introduction,
on travaillait au-dessus d'une structure $A$ particulière, mais celle-ci ne jouait aucun
rôle dans la description évoquée.}, à 
partir d'un quotient de deux foncteurs en groupes $\emptyset$-définissables
dans le langage classique $\mathscr L_{\rm val}$.
On vérifie
sans peine que si $F\in \mathsf M_\emptyset$, l'application
de $F\times \abs F$ vers $\overline B(F)$ qui envoie $(\lambda,r)$ sur
$b(\lambda,r)_F$ est $\emptyset$-définissable.

\section{Types}

{\em Références :} pour ce qui concerne les types, nous renvoyons le lecteur aux
premiers chapitres de l'article étudié
\cite{hrushovski-loeser2010}, à l'ouvrage \cite{haskell-hrush-macph2008}, 
et aux transparents d'exposés de Hrushovski 
(\cite{notes-hrush2011}). 

\subsection{Généralités}\label{typegeneral}

On fixe un langage $\mathscr L$ et une théorie $T$ dans
le langage $\mathscr L$ ; on suppose
que $T$ élimine les quantificateurs. Soit $A$
une structure de $\mathscr L$, et soit $D$
un foncteur $A$-définissable. Soient $N$ et $N'$ deux 
objets de $\mathsf M_A$. Soit $x\in D(N)$ et soit
$x'\in D(N')$. On dit que $(N,x)$ et $(N',x')$ sont
{\em $A$-équivalents} si pour tout sous-foncteur
$A$-définissable $\Delta$ de $D$, on a $$x\in \Delta(N)\iff x'\in \Delta(N')$$
(autrement dit, $x$ et $x'$ ne sont pas discernables par une formule
à paramètres dans~$A$). 
Un {\em type 
sur $D$}
est une classe de $A$-équivalence de couples $(N,x)$ comme ci-dessus. 
Si $t$~est la classe de $(N,x)$, on dira que $x$ {\em réalise} le type $t$. 
Soit $t$ un type sur $D$ et soit 
$x$ une réalisation de $t$, appartenant à un modèle $N$. 
Soit $\Delta$ un sous-foncteur $A$-définissable de $D$. 
L'appartenance
ou non de $x$ à $\Delta(N)$ ne dépend que de $t$ ; si elle est avérée,
on dira que $t$ est situé sur $\Delta$. Il n'y a pas de conflit de terminologie : 
le couple $(N,x)$ définit justement dans ce cas sans ambiguïté 
un type sur $\Delta$. Notons que le type 
$t$ est par définition {\em caractérisé}
par l'ensemble des 
sous-foncteurs $A$-définissables
de $D$ sur lesquels il est situé.  
 On note
$S(D)$ l'ensemble\footnote{Soit $\mathscr F$ l'ensemble
des sous-foncteurs $A$-définissables de $D$ (qui s'identifie, rappelons-le,
à l'ensemble
des parties $A$-définissables de $D(M)$ pour tout
$M\in \mathsf M_A$). On peut par ce qui
précède voir un type sur $D$ comme une application de $\mathscr F$ vers
$\{0,1\}$ ; par conséquent, $S(D)$ est bien un ensemble.}
des types sur $D$. 

Soit $D'$ un foncteur $A$-définissable et soit $f: D\to D'$
une transformation naturelle définissable. Soit $t\in S(D)$ 
et soit $x$ une réalisation
de $t$ ; son image $f(x)$ induit un type sur $D'$ qui ne dépend que de $t$, 
et que l'on note $f(t)$. 

\medskip
Supposons que $A$ soit un modèle de $T$. Si $x\in D(A)$ il définit un type 
sur $D$ (dont il est une, et même la seule, réalisation)
; un tel type est qualifié de {\em simple}. 
Donnons maintenant quelques exemples plus intéressants. 

\medskip
(1) On suppose que $\mathscr L={\mathscr L}_{\rm ann}$, et que $T$ 
est la théorie de corps algébriquement clos. Soit $F$ un corps et soit
$X$ un $F$-schéma de type fini ; on peut le voir comme un foncteur $F$-définissable. 
Soit $x$ un point du schéma $X$ et soit $F'$ une extension algébriquement close
du corps résiduel $F(x)$. Le point $x$ définit un point de $X(F')$, et 
partant un type 
sur $X$. Cette construction induit une {\em bijection} entre l'ensemble
sous-jacent au schéma~$X$ et $S(X)$.

À titre d'illustration, supposons que $X$ est intègre, et soit $x$ son point générique. 
Le type correspondant
est alors caractérisé par le fait qu'il n'est situé sur {\em aucun} fermé de Zariski
strict de $X$. 

\smallskip

(2) On suppose que $\mathscr L={\mathscr L}_{\rm co}$, et que $T$ 
est la théorie des corps réels clos. 
Soit $R$ un corps ordonné, et soit $X$ un $R$-schéma
de type fini ; on peut le voir comme un foncteur $R$-définissable. 
Rappelons que le
{\em spectre réel} $X_r$ de $X$ est l'ensemble des couples $(x,\leq)$ où 
$x$ est un point du schéma $X$ et $\leq $ un ordre sur le corps résiduel $R(x)$
prolongeant l'ordre sur $R$.  
Soit $\xi=(x,\leq)$ un point de $X_r$ et soit $R'$ la clôture réelle de $(R(x),\leq)$. 
Le point $x$ définit un point de $X(R')$, et 
partant un type
sur $X$. Cette construction induit une {\em bijection} entre l'ensemble
$X_r$ et $S(X)$. 

À titre d'illustration, supposons que $X={\mathbb A}^1_R$, et soit
$u$ la fonction coordonnée sur~$X$. Munissons $R(u)$ de l'ordre
tel que $u>0$ et $u<\epsilon$ pour tout élément $\epsilon>0$ de $R$. On définit
par ce biais un point de $X_r$, supporté par le point générique de~$X$. 
Le type 
correspondant est noté $0^+_R$ ; il est caractérisé
par le fait qu'il est situé sur $]0;\epsilon[_R$ pour tout élément $\epsilon>0$ de $R$. 

\smallskip

(3) On suppose que $\mathscr L={\mathscr L}^*_{\rm val}$, et que $T$ 
est la théorie ACVF. Soit $F$ un corps valué, et soit $X$ un $F$-schéma
de type fini ; on peut le voir comme un foncteur $F$-définissable. Rappelons que le
{\em spectre valuatif} $X_v$ de $X$ est l'ensemble des couples $(x,\abs .)$ où 
$x$ est un point du schéma $X$ et $\abs. $ une valuation sur le corps résiduel $F(x)$,
qui prolonge la valuation de $F$. 
Soit $\xi=(x,\abs .)$ un point de $X_v$ et soit $F'$ une extension
non trivialement
valuée et algébriquement close de $(F(x),\abs .)$. 
Le point $x$ définit un point de $X(F')$, et 
partant un type sur $X$. Cette construction induit une {\em bijection} entre l'ensemble
$X_v$ et $S(X)$. 

À titre d'illustration, supposons que $X={\mathbb A}^1_F$, et soit
$u$ la fonction coordonnée sur~$X$. Soit $\lambda\in F$ et soit
$r\in \abs F$. Munissons $F(u)$ de la valuation
$\abs. $ telle que l'on ait \mbox{$\abs {\sum a_i (u-\lambda)^i}=\max \abs{a_i}r^i$}
 pour tout polynôme 
$\sum a_i (u-\lambda)^i$
à coefficients dans $F$. On définit
par ce biais un point de $X_v$, supporté par le point générique de $X$ ; 
on note  $\eta_{\lambda,r,F}$ le type correspondant. 
Si $F$ est un modèle de ACVF, le type
$\eta_{\lambda, r,F}$ est caractérisé par 
le fait qu'il est situé sur $b(\lambda,r)_F$, mais qu'il
n'est situé sur 
aucune boule ferm\'ee ou ouverte $F$-d\'efinissable contenue
{\em strictement} dans $b(\lambda,r)_F$ ; 
il s'ensuit aisément que $\eta_{\lambda,r,F}=\eta_{\mu, s,F}$
si et seulement si $b(\lambda,r)_F=b(\mu,s)_F$. On 
appelle parfois $\eta_{\lambda,r,F}$
le {\em type générique de $b(\lambda,r)_F$}.

Supposons que $r>0$. Donnons-nous
maintenant un groupe abélien ordonné contenant $|F^*|$
ainsi qu'un élément $\omega$ tel que $\omega<1$ et $\alpha<\omega$
pour tout $\alpha\in \abs{F\zeroo}$. 
Munissons $F(u)$ de la valuation
$\abs. $ telle que l'on ait~$\abs {\sum a_i (u-\lambda)^i}=\max \abs{a_i}r^i\omega^i$ pour tout polynôme 
$\sum a_i (u-\lambda)^i$
à coefficients dans $F$. On définit
par ce biais un point de $X_v$, supporté par le point générique de $X$ ; 
on note  $\eta_{\lambda,r^-,F}$ le type correspondant. 
Lorsque $F$ est un modèle de ACVF, 
le type~$\eta_{\lambda,r^-,F}$ est  caractérisé par le fait qu'il est situé sur 
$b^{\rm ouv}(\lambda,r)_F$ 
mais qu'il
n'est situé sur aucune boule fermée~$\beta\in \overline B(F)$ 
contenue
dans $b^{\rm ouv}(\lambda,r)_F$ ; il s'ensuit aisément que $\eta_{\lambda,r^-,F}=\eta_{\mu, s^-,F}$
si et seulement si $b^{\rm ouv}(\lambda,r)_F=b^{\rm ouv}(\mu,s)_F$.
On appelle
parfois $\eta_{\lambda,r^-,F}$
le {\em type générique de $b^{\rm ouv}(\lambda,r)_F$}. 

\smallskip

(4)  On travaille toujours
avec le langage $\mathscr L^*_{\rm val}$ ; soit $k$ un corps 
ultramétrique complet. 
Soit  $X$ un $k$-schéma
de type fini ; on peut le voir comme un foncteur \mbox{$k$-définissable.} Rappelons que 
{\em l'analytifié} $X^{\rm an}$ de $X$ (au sens de Berkovich)
est l'ensemble des couples $(x,\abs .)$ où 
$x$ est un point du schéma $X$ et où $\abs .$ est une valeur
absolue ultramétrique sur le corps résiduel $k(x)$. 
Soit $\xi=(x,\abs .)$ un point de $X^{\rm an}$ et soit $k'$ une extension
ultramétrique non trivialement valuée et algébriquement close 
de $(k(x),\abs .)$. 
Le point $x$ définit un point de $X(k')$, et 
partant un type sur $X$ qui est {\em ultramétrique}, au sens
où il admet une réalisation dans un modèle ultramétrique de ACVF. 
Cette construction induit une {\em bijection} entre l'ensemble
$X^{\rm an}$ et l'ensemble des types ultramétriques situés sur $X$. 

%\end{enumerate} 
\subsection{Topologies sur les espaces de types}\label{toptypes}

On conserve les notations $\mathscr L$ et $T$ introduites 
au début de~\ref{typegeneral}. 
Soit $A$ une structure de $\mathscr L$ et soit 
$D$ un foncteur $A$-définissable. 
On munit $S(D)$  
de la topologie dont une base d'ouverts est constituée des $S(\Delta)$,
où $\Delta$ parcourt l'ensemble des sous-foncteurs $A$-définissables de $D$. 

{\em Reprenons l'exemple (1) de \ref{typegeneral}}. L'ensemble $S(X)$
s'identifie au schéma $X$. La topologie sur $X$ déduite de celle de
$S(X)$ est alors la topologie
{\em constructible} (celle qui est engendrée par les ouverts {\em et les fermés} de
Zariski). 

{\em Reprenons l'exemple (2) de \ref{typegeneral}}. On a vu que  $S(X)$
s'identifie au spectre réel $X_r$ de~$X$. 
La topologie sur $X_r$ déduite de celle de
$S(X)$ est alors la topologie
{\em constructible} (celle qui est engendrée Zariski-localement
par les combinaisons booléennes  d'inégalités strictes
{\em ou larges} entre fonctions régulières).

\medskip
La topologie ainsi définie sur $S(D)$ a l'avantage d'en faire
un espace topologique compact (nous discuterons cette propriété plus avant
un peu plus bas). Elle a l'inconvénient de faire de
{\em tout} définissable un ouvert, indépendamment de la nature
de la formule qui le décrit. Dans certaines circonstances, 
où l'on estime que certaines formules
«doivent» être ouvertes et d'autres non, on peut donc
être amené
à introduire une topologie alternative en prenant pour base
d'ouverts les $S(\Delta)$,
pour $\Delta$ parcourant
{\em un certain sous-ensemble} de 
l'ensemble des sous-foncteurs $A$-définissables de $D$. 

\medskip
{\em Reprenons l'exemple (1) de \ref{typegeneral}}. L'ensemble $S(X)$
s'identifie au schéma $X$. On peut munir $S(X)$ de la topologie
engendrée par les $S(\Delta)$, où $\Delta$ est un ouvert
de Zariski de $X$. La topologie ainsi définie sur $S(X)$ correspond
évidemment à 
la topologie de Zariski de $X$. 

{\em Reprenons l'exemple (2) de \ref{typegeneral}}. On a vu que  $S(X)$
s'identifie au spectre réel $X_r$ de $X$.  On peut munir $S(X)$ de la topologie
engendrée par les $S(\Delta)$, 
où $\Delta$ est défini Zariski localement
par combinaison booléenne positive
d'inégalités {\em strictes} entre fonctions 
régulières. 
La topologie sur $X_r$
à laquelle elle correspond
est la topologie usuelle du spectre réel. 

\subsection{Digression sur la compacité}\label{digression}

La compacité des espaces de type joue un rôle
majeur en théorie des modèles en général, et chez  Hrushovski et
Loeser en particulier, où elle intervient à maints endroits -- même
si cela n'apparaîtra guère dans ce texte, où nous contenterons de brosser
très grossièrement les preuves. 
Le plus souvent, elle s'utilise sous la forme du principe suivant, à l'énoncé volontairement
vague : si $\mathscr P$ est une propriété à l'énoncé raisonnable
dans le langage $\mathscr L$, il revient
au même de la démontrer en situation {\em relative} 
sur {\em un} modèle $M$
donné ou en situation {\em absolue} sur {\em tous} les modèles $M'$ tels que $M\subseteq M'$. 
C'est ce qui explique qu'il puisse être nécessaire, même si l'on s'intéresse à un modèle $M$
bien défini, de travailler sur des modèles éventuellement beaucoup plus gros. 
Nous allons étayer notre propos à travers deux exemples. 

\medskip
Le premier concerne  la {\em trivialisation constructible d'un faisceau cohérent}. 
Soit $F$ un corps
et soit $X$ un $F$-schéma de type fini. Soit $\mathscr F$ un
faisceau cohérent sur $X$. Il existe alors une partition constructible
finie et localement fermée
$(X_i)$ de $X$ telle que $\mathscr F|_{X_i}$ soit libre pour tout~$i$ 
(où $X_i$ est muni de sa structure réduite). En effet, si 
$x$ est un point du  schéma $X$ alors $\mathscr F\otimes F(x)$ est libre
(c'est la forme absolue du résultat à établir, sur le corps $F(x)$),
et cette propriété se propage au-dessus d'un ouvert de Zariski non vide
de $\overline{\{x\}}_{\rm red}$ ; on conclut en invoquant la compacité
de $X$ pour la topologie constructible -- ou une récurrence
noethérienne si l'on préfère, ce qui revient au même. 

On voit que si l'on s'était contenté de ne considérer que des $F$-points
de $X$, on n'aurait pas pu aboutir m\^eme si $F$ est alg\'ebriquement clos~: 
si $x\in X(F)$, la liberté de $\mathscr F\otimes F(x)$
n'a aucune raison {\em a priori} de se propager au-delà du singleton 
constructible~$\{x\}$. 

\medskip
Le second concerne {\em le théorème de Hardt} ; 
nous allons 
donner les grandes lignes de la preuve 
qu'en ont proposée Jacek Bochnak, Michel Coste
et Marie-Françoise Roy (\cite{bochnak-coste-roy1987}, th. 9.3.1)
ou disons plutôt 
d'une traduction de celle-ci dans le langage des types
(ils utilisent quant à eux celui du spectre réel). 
Soit $R$ un corps réel clos, et 
soient $Y\subset {\mathbb A}^n_R$ et $X\subset {\mathbb A}^m_R$ deux
foncteurs
$R$-définissables dans le langage 
des corps ordonnés (on dit aussi qu'ils sont
{\em semi-algébriques}). 
Soit $f: Y\to X$
une transformation naturelle 
\mbox{$R$-définissable.} On la suppose continue dans le
sens suivant : pour tout corps réel clos
$R'$ contenant $R$, l'application $f(R'): Y(R')\to Y(R')$
est continue pour les topologies  déduites 
de l'ordre sur $R'$ (il suffit en vertu 
du lemme \ref{modcomp}
de le tester sur {\em un} tel corps
$R'$, par exemple sur $R$). Il existe alors une partition finie
et $R$-définissable $(X_i)$ de $X$ et, pour tout~$i$, un
sous-foncteur $R$-définissable $Z_i$ de ${\mathbb A}^n_R$ telle que $f^{-1}(X_i)$ soit
$R$-définissablement homéomorphe à $Z_i\times X_i$.
En effet, soit $t\in S(X)$ et soit $x$ une réalisation de $t$, 
appartenant à $X(S)$ pour un certain corps réel clos $S$ contenant $R$. 
La fibre $f_S^{-1}(x)$
est un sous-foncteur $S$-définissable de ${\mathbb A}^n_S$, 
et possède à ce titre
une triangulation \mbox{$S$-définissable}\footnote{Cela signifie
qu'elle s'identifie à une union finie de cellules {\em ouvertes} 
d'un complexe simplicial (elle n'est
pas forcément fermée bornée) ; l'expression «ensemble simplicial»
que nous employons ensuite fait référence à la donnée combinatoire 
de ces cellules ouvertes.}. 
L'ensemble simplicial
correspondant 
admettant une réalisation dans ${\mathbb A}^m_S$, il en admet
une, disons $Z$, dans ${\mathbb A}^m_R$ (par le lemme
\ref{modcomp}). Par conséquent, il existe
un homéomorphisme $S$-définissable
$f_S^{-1}(x)\simeq Z_S$ (c'est la forme absolue du résultat à établir, 
sur le corps réel clos $S$). Cette propriété se
propage : il existe un sous-foncteur $R$-définissable
$T$ de $X$ tel que
$t \in S(T)$ et tel que $f^{-1}(T)$ soit $R$-définissablement
homéomorphe à $ Z\times T$. On conclut
en utilisant la compacité de $S(X)$. 

On voit que si l'on s'était contenté de ne considérer que des $R$-points,
on n'aurait pas pu aboutir : le type d'homéomorphie semi-algébrique 
de la fibre en 
un tel point $x$ n'a aucune raison {\em a priori} de se propager
au-delà du singleton $\{x\}$. 

\medskip
\begin{rema}\label{compultraf} Soit~$M$ un modèle de~$T$,
et soit~$D$ un foncteur~$M$-définissable. Soit~$x\in S(D)$, 
et soit~$\mathscr  E_x$
l'ensemble des sous-foncteurs~$M$-définissables
de~$D$ sur lesquels~$x$ est situé. L'ensemble~$\mathscr U_x$
des parties~$M$-définissables de~$D(M)$ de la
forme~$\Delta(M)$, où~$\Delta\in \mathscr E$, est un ultra-filtre
de parties~$M$-définissables. Réciproquement, 
il résulte de la compacité de~$S(D)$ que
tout ultrafiltre de parties~$M$-définissables de~$D(M)$ 
est de la forme~$\mathscr U_x$ pour un certain~$x\in S(D)$ 
(nécessairement unique par définition d'un type). 
\end{rema}

\subsection{Types définissables}

On conserve les notations $\mathscr L$ et $T$ introduites 
au début de~\ref{typegeneral}. 
\begin{defi}\label{deftypedef} Soit $M$ un modèle de $T$, soit $D$
un foncteur $M$-définissable et soit $t\in S(D)$. On dit que
$t$  est {\em $M$-définissable} s'il possède la propriété suivante : pour tout foncteur
$M$-définissable $D'$ et tout sous-foncteur 
$M$-définissable $\Delta\subset D\times D'$,
le sous-ensemble de $D'(M)$ formé des points $x$ tels
que $t$ soit situé sur $\Delta\times_{D'}\{x\}\subset D$ est $M$-définissable. 

\end{defi} 

Donnons quelques exemples et contre-exemples. 

\medskip
$\bullet$ Plaçons-nous dans la théorie des corps algébriquement clos, et
soit $D$ un foncteur $F$-définissable pour un certain corps
algébriquement clos $F$.  Tout type sur $D$ est alors $F$-définissable. 

$\bullet$ Plaçons-nous dans la théorie des corps
réels clos et soit $R$ un corps réel clos. 
Le type $0^+_R$ défini à la fin de l'exemple (2) de~\ref{typegeneral}
est $R$-définissable. Cela provient essentiellement du fait suivant. Soit 
$f\in R(u)$ et soit $D$ le sous-foncteur de ${\mathbb A}^1_R$ défini par
la condition $f\Join 0$, ou $\Join$ est un symbole 
appartenant à $\{<,>,\leq,\geq\}$. 
Le type $0^+_R$ est alors situé sur $D$ si et seulement si il existe 
$\epsilon >0$ dans $R$ tel que $f(x)\Join 0$ 
pour tout $x$ tel que $0<x<\epsilon$,
et cette condition s'exprime visiblement par une 
formule de $\mathscr L_{\rm co}$ à paramètres dans $R$ et portant
sur les coefficients de $f$. 

$\bullet$ Plaçons-nous dans la théorie ACVF et soit
$F$ un modèle de celle-ci.  
Les types $\eta_{\lambda,r,F}$ et $\eta_{\lambda,r^-,F}$ 
définis à l'exemple (3) de \ref{typegeneral} sont
$F$-définissables : cela résulte essentiellement de leurs descriptions
par des formules de $\mathscr L^*_{\rm val}$ à paramètres dans $F$. 

$\bullet$ Plaçons-nous dans la théorie des corps réels clos, et soit
$R_0$ la fermeture algébrique de $\mathbb Q$ dans $\mathbb R$ 
(c'est un corps réel clos). Le nombre réel $\pi\in {\mathbb A}^1({\mathbb R})
$ définit un type sur ${\mathbb A}^1_{R_0}$ ; il n'est pas
$R_0$-définissable : si $f\in R_0(u)$, le signe de $f(\pi)$ ne s'exprime pas par une
formule de $\mathscr L_{\rm co}$ 
à paramètres dans $R_0$ en les coefficients de $f$. 

$\bullet$ On reste dans la théorie  des corps réels clos. 
On se donne un corps réel clos $R$ contenant strictement $\mathbb R$
(il existe alors dans $R$ un élément supérieur à tout réel 
 positif). Munissons 
$R(u)$ de l'ordre pour lequel un élément $f$ est strictement
positif si et seulement si il existe $N\in \mathbb N$ tel que
$f(x)$ soit strictement positif pour tout $x$ {\em appartenant à $\mathbb R$}
et strictement supérieur à $N$. On définit
ainsi un point du spectre réel de ${\mathbb A}^1_R$, et partant un type 
sur ${\mathbb A}^1_R$. Ce type n'est pas $R$-définissable ; 
cela
résulte essentiellement du fait que $\mathbb R$ n'est pas une partie 
$R$-définissable
de $R$. 

\medskip
{\em Restriction des types, extension des types
définissables.} Soit $M$ un modèle de $T$, 
soit $D$ un foncteur $M$-définissable, et soit $M'\in \mathsf M_M$. 
Soit $\theta$ un type
sur $D_{M'}$. Si $x$ désigne
une réalisation de $\theta$, alors $x$ induit un type
sur $D$, qui ne dépend que de $\theta$, et pas du choix
de $x$ ; on dispose ainsi
d'une application naturelle $S(D_{M'})\to S(D)$,
parfois dite de {\em restriction au modèle $M$}.

Soit maintenant $t$ un type {\em $M$-définissable} situé sur $D$. 
Il existe
alors un {\em antécédent 
canonique} $t'$ de $t$ sur $S(D_{M'})$
qui est $M'$-définissable 
(on dit que
c'est {\em l'extension canonique de $t$ au modèle $M'$}) ; 
grossièrement, le
type~$t'$ est défini par les mêmes
formules que $t$. Plus précisément, soit $\Delta$
un sous-foncteur $M'$-définissable de $D_{M'}$. En
«faisant varier les paramètres de la formule qui définit $\Delta$», 
on montre l'existence d'un foncteur $M$-définissable $D_1$, 
d'un sous-foncteur
$M$-définissable $D_2$ de $D\times D_1$, et d'un point
$x\in D_1(M')$ tel que $\Delta=D_{2,M'} \times_{D_{1,M'}}\{x\}$. Le type $t$ étant
$M$-définissable, il existe un sous-foncteur
$M$-définissable $D_3$ de $D_1$ tel que 
pour tout $y\in D_1(M)$, le point $y$ appartienne
à $D_3(M)$ si et seulement si 
$t$ est situé sur $D_2\times_{D_1}\{y\}$. 
Le type $t'$ est alors 
situé sur $\Delta$ si et seulement si 
$x\in D_3(M')$.

Donnons maintenant quelques exemples d'extension canonique. 

$\bullet$ Si $x\in D(M)$ il définit un type simple $t$ sur $D$ ; l'extension
canonique de $t$ au modèle $M'$ est le type simple sur $D_{M'}$
associé à $x\in D(M)\subset D(M')$. 

$\bullet$ On reprend les notations de l'exemple (1) de~\ref{typegeneral}, en supposant
que $F$ est algébriquement clos. Soit
$F'$ une extension algébriquement close de $F$, soit $t\in S(X)$,
et soit $t'\in S(X_{F'})$ son extension canonique. Le type $t$ correspond
à un point $x$ du schéma $X$. Le fermé de Zariski
$\overline{\{x\}}_{F'}$ de $X_{F'}$ est irréductible (car $F$ est algébriquement clos). 
Son point générique $x'$ est précisément le point du schéma $X_{F'}$ qui correspond
à~$t'$. 

$\bullet$ On reprend les notations de l'exemple (2) de~\ref{typegeneral}
en supposant que $R$ est réel clos. Soit  $R'$
une extension réelle close de $R$. L'extension canonique du type $0^+_R$ au modèle
$R'$ est le type $0^+_{R'}$. 

$\bullet$ On reprend les notations de l'exemple (3) de~\ref{typegeneral} en 
supposant que $F$ est un modèle de ACVF. 
Soit  $F'$
une extension valuée algébriquement close de $F$. 
L'extension canonique du type $\eta_{\lambda,r,F}$ (resp. $\eta_{\lambda,r^-,F}$) au modèle
$F'$ est le type $\eta_{\lambda,r,F'}$ (resp. $\eta_{\lambda,r^-,F'}$).

\medskip
Terminons ce paragraphe
en mentionnant que si $f: D\to D'$ est une transformation naturelle $M$-définissable
entre deux foncteurs $M$-définissables, et si $t$ est un type 
\mbox{$M$-définissable}
sur $D$, alors
le type image $f(t)$ est $M$-définissable. 

\subsection{Topologies définissables ; compacité définissable}

On conserve les notations $\mathscr L$ et $T$ introduites 
au début de~\ref{typegeneral}. Soit $A$ une structure de $T$, 
et soit $D$ un foncteur $A$-définissable. Une {\em topologie définissable}
$\mathscr T$ sur $D$ consiste en la donnée, pour tout $M\in \mathsf M_A$, d'une
famille $\mathscr T_M$ de sous-foncteurs $M$-définissables de $D_M$, 
que l'on appelle les {\em ouverts $M$-définissables de $D_M$}, possédant les 
propriétés suivantes\footnote{On impose également
des conditions de finitude très raisonnables sur lesquelles nous ne nous étendrons
pas ici} : 

$\bullet$ la famille $\mathscr T_M$ est stable
par intersections finies ; 

$\bullet$ si $(U_i)$ est une famille
d'éléments de $\mathscr T_M$ telle  
que $\bigcup U_i(M)=V(M)$ pour un certain sous-foncteur $M$-définissable
$V$ de $D_M$ alors $V\in \mathscr T_M$ ;  

$\bullet$ si $M\subseteq M'$ et si $U$ est un sous-foncteur
$M$-définissable de $D_M$ alors~$U\in \mathscr T_M$ 
si et seulement si $U_{M'}\in \mathscr T_{M'}$. 

Soit $\mathscr T$ une topologie définissable sur $D$. Elle
induit pour tout $M$
une topologie~$\mathscr T(M)$ sur $D(M)$, à savoir celle engendrée par les $U(M)$
pour $U$ parcourant  $\mathscr T_M$. Lorsque~\mbox{$M\subseteq M'$,} la topologie $\mathscr T(M)$
est plus grossière (et, en général, {\em strictement} plus grossière)
que la topologie induite par celle $\mathscr T(M')$. Si $B$ est une structure telle
que $A\subseteq B$, on dira qu'un sous-foncteur $B$-définissable $U$ de $D_B$ est 
un {\em ouvert $B$-définissable} si $U_M$ est un ouvert définissable de $D_M$
pour tout $M\in \mathsf M_B$ -- il suffit que ce soit le cas pour {\em un} tel~$M$. 

Si $D'$ est un foncteur $A$-définissable muni d'une topologie définissable $\mathscr T'$,
une transformation naturelle $A$-définissable $f:D\to D'$ est dite {\em continue} si pour
tout $M\in \mathsf M_A$ et tout $U\in \mathscr T'_M$, le foncteur $f^{-1}(U)$ appartient à 
$\mathscr T_M$. Il revient au même de demander que $f(M) : D(M)\to D'(M)$
soit continue pour {\em tout} modèle $M\in \mathsf M_A$.

Soit $M\in \mathsf M_A$, soit $t$ un type sur $D_M$ et soit 
$x\in D(M)$. On dit que $x$ {\em adhère à $t$} si pour 
tout $U\in \mathscr T_M$ tel que $x\in U(M)$, 
le type $t$ est situé sur $U$.

On dira que $D$ est {\em définissablement compact}
si pour tout 
$M\in \mathsf M_A$ et tout type {\em $M$-définissable} $t$ sur $D_M$, 
il existe un unique $x\in D(M)$ qui est adhérent à $t$. Cela revient en quelque sorte
à demander que «tout ultra-filtre~$M$-définissable de
parties~$M$-définissables de~$D(M)$ converge», {\em cf.} rem.~\ref{compultraf}.  

Si $D$ est définissablement
compact et si $\Delta$ est un sous-foncteur $A$-définissable de $D$, alors
$\Delta$ est un fermé $A$-définissable ({\em i.e.} le complémentaire d'un ouvert
$A$-définissable) si et seulement si il est définissablement compact. 

\begin{exem}\label{topdef} On se place dans la théorie des corps 
réels clos. Soit $R_0$ un corps réel clos et soit $X$ un $R_0$-schéma de type fini. 
Pour tout corps réel clos $R$ contenant $R_0$, notons 
$\mathscr T_R$ la famille des sous-foncteurs $R$-définissables
de $X_R$ définis Zariski-localement sur $X_R$ par une combinaison
booléenne positive d'inégalités {\em strictes} entre fonctions régulières. La donnée
des $\mathscr T_R$ constitue une topologie définissable sur $X$ et en induit une, 
par restriction, sur tout sous-foncteur $R_0$-définissable de $X$. Si $R$ est un
corps réel clos contenant $R_0$, la topologie $\mathscr T(R)$ 
sur $X(R)$ est celle déduite de la topologie de corps ordonné de $R$. Remarquons
que si $R'$ est un corps réel clos contenant $R$ et s'il existe un élément strictement
positif $\epsilon$ de $R'$ qui est inférieur à $x$ pour tout $x>0$ dans $R$, la topologie
de $X(R)$ induite par $\mathscr T(R')$ est la topologie {\em discrète} : en effet, 
pour tout $x\in R$ on a $\{x\}=\{y\in R',-\epsilon <y-x<\epsilon\}$. 

Si $X={\mathbb A}^n_{R_0}$ et si $Y$ est un sous-foncteur
$R$-définissable de $X$ alors $Y$ est définissablement compact 
si et seulement si il est fermé (autrement dit, $X\setminus Y\in \mathscr T_{R_0}$)
et borné, c'est-à-dire contenu dans $[-x;x]_{R_0}^n$ pour un certain $x>0$ dans
$R_0$. 

Donnons un exemple de point adhérent à un type. On suppose maintenant
que~$X$ est égal à ${\mathbb A}^1_{R_0}$. Le point $0\in {\mathbb A}^1(R_0)=R_0$ est alors
adhérent au type $0_{R_0}^+$ (et c'est le seul point de $R_0$ dans ce cas).
On voit ainsi pourquoi 
$]0;1[_{R_0}$ n'est pas définissablement compact : $0_{R_0}^+$ est situé sur $]0;1[_{R_0}$,
mais aucun~$R_0$-point de ce dernier n'adh\`ere \`a $0^+_{R_0}$. 
Insistons à ce propos sur l'importance, 
pour la compacité définissable, de se limiter à manipuler des 
types {\em définissables}. Pour
s'en convaincre, il suffit de considérer le cas où $R_0$ est la fermeture algébrique de $\mathbb Q$ 
dans $\mathbb R$. Le sous-foncteur $R_0$-définissable
$[0;4]_{R_0}$ de $\mathbb A^1_{R_0}$ est définissablement
compact. Le nombre réel $\pi$ induit un type $t$ sur $[0;4]_{R_0}$
qui n'est pas $R_0$-définissable... et de fait, aucun $R_0$-point 
de $[0;4]_{R_0}$ n'adh\`ere \`a $t$.

\end{exem}

\subsection{Types stablement dominés ; le cas de ACVF}\label{typstabdom}

Commençons par une remarque, qui sera implicitement 
utilisée dans toute la suite du texte. Soit $F$ un modèle de ACVF, et
soit $D$ un foncteur $\abs F$-définissable dans la théorie des
groupes abéliens divisibles ordonnés non triviaux (resp. un foncteur $\tilde F$-définissable
dans la théorie des corps algébriquement clos). Soit $D'$ le foncteur 
qui envoie $L\in \mathsf M_F$ sur $D(\abs L)$ (resp. $D(\tilde L)$). Il
résulte alors de la remarque~\ref{stablyemb} qu'il existe une bijection canonique
$S(D)\simeq S(D')$, 
et qu'un type sur $D'$ est $F$-définissable
si et seulement si le type correspondant sur $D$ est $\abs F$-définissable
(resp. $\tilde F$-définissable). 

\medskip
Venons-en maintenant à la stabilité. 
Il n'est pas question de donner des définitions précises
en la matière ; nous renvoyons le lecteur intéressé
à l'ouvrage \cite{haskell-hrush-macph2008} de Haskell, Hrushovski 
et Macpherson. Nous allons nous contenter de quelques
explications très sommaires. 

1) Il existe différentes définitions
équivalentes d'une théorie {\em stable}. L'une d'elles
consiste à exiger «qu'elle ne contienne pas d'ensemble ordonné infini» ; 
une autre requiert que pour tout modèle $M$ et tout 
foncteur $M$-définissable $D$, le cardinal de $S(D_{M'})$
croisse raisonnablement en fonction de celui du modèle $M'\supseteq M$. La théorie
des corps algébriquement clos est stable. La théorie des 
groupes abéliens divisibles ordonnés non triviaux,
la théorie des corps réels
clos et la théorie ACVF ne sont pas stables : chacune
d'elles «contient un ensemble ordonné infini». 

2) Une théorie $T$ possède toutefois une sorte de
«plus grande sous-théorie stable» $T_{\rm stab}$. 
Si $T$ est la théorie des
groupes abéliens divisibles ordonnés non triviaux, 
$T_{\rm stab}$ ne voit que
les foncteurs définissables {\em finis} 
(c'est-à-dire dont l'ensemble des points
à valeur dans un modèle donné, et donc dans tout modèle, 
est fini). 
Si $T={\rm ACVF}$, la théorie $T_{\rm stab}$ est essentiellement
la théorie de la sorte «corps résiduel» 
(qui s'identifie à celle des corps algébriquement clos). 

3) Un type définissable dans une 
théorie $T$ est dit {\em stablement dominé} s'il est déterminé
par sa trace sur la plus grande sous-théorie stable
de $T$. 

\medskip
Un type simple est toujours stablement dominé. 
La réciproque est parfois vraie. Ainsi, si $T$ est la théorie des
groupes abéliens divisibles ordonnés non triviaux, et si 
$D$ est un foncteur $M$-définissable pour un certain modèle
$M$ de $T$, un type stablement dominé sur $D$ est un type
$M$-définissable
$t$ qui doit être caractérisé par $f(t)$ pour une certaine transformation
naturelle $M$-définissable de $D$ vers un foncteur
$M$-définissable {\em fini} $D'$ ; il n'est pas difficile
de voir que cela équivaut à demander
que $t$ soit {\em simple}. La théorie des types stablement
dominés est donc triviale dans la théorie $T$.

Il n'en va pas de même dans la théorie ACVF. Donnons un exemple de type
stablement dominé qui n'est pas simple. On fixe un modèle $F$. On dispose
d'une transformation naturelle $F$-définissable $\rho$ de $b(0,1)_F$
 vers $\mathbb{A}^1_{\tilde F}$
(la réduction modulo l'idéal maximal). Soit $t\in S(  b(0,1)_F)$. 
On vérifie aisément
que $t=\eta_{0,1,F}$ si et seulement si $\rho(t)$
est le type correspondant au point générique de $\mathbb{A}^1_{\tilde F}$. 
Comme $\mathbb{A}^1_{\tilde F}$
«appartient» à ${\rm ACVF}_{\rm stab}$, le type $\eta_{0,1,F}$ est contrôlé par
sa trace sur  ${\rm ACVF}_{\rm stab}$, et est dès lors
stablement dominé.  

En fait, on dispose dans la théorie ACVF d'un critère permettant de décider si
un type définissable est stablement dominé, que nous allons
maintenant énoncer ; on rappelle 
que $\Gamma_0$ désigne le foncteur $F\mapsto \abs F$.

\begin{prop}\label{equivstabdom} Soit $F$ un
modèle de ACVF, soit $D$ un foncteur $F$-définissable
et soit $t$ un type $F$-définissable sur $D$. Le type $t$ est 
stablement
dominé si et seulement si il est {\em orthogonal à $\Gamma_0$},
c'est-à-dire si pour tout
modèle $F'\supseteq F$ et toute transformation 
naturelle $F'$-définissable
$f: D_{F'}\to \Gamma_{0,F'}$, le type image $f(t)$ est simple. 
\end{prop}

\begin{rema}\label{remstabdom} 
Supposons que $D$ soit un sous-foncteur d'un $F$-schéma de type
fini~$X$. Le type $t$ correspond alors à un point $(x,\abs.)$ du spectre
valuatif de $X$, et la condition énoncée dans la proposition
ci-dessus équivaut
simplement
à l'égalité $\abs{F(x)}=\abs F$. 
\end{rema}

\begin{rema}\label{stabdomfonct}
Soit $F$ un
modèle de ACVF, soit $D$ un foncteur $F$-définissable, et soit 
$t$ un type stablement dominé sur $D$. Il résulte immédiatement
de la proposition ci-dessus que : 

$\bullet$ si $f$ est une transformation naturelle $F$-définissable de 
$D$ vers un foncteur \mbox{$F$-définissable,} le type $f(t)$ est stablement dominé ; 

$\bullet$ si $F'\in \mathsf M_F$, l'extension canonique de $t$ au modèle $F'$
est stablement dominée.

\end{rema}

On a vu plus haut que $\eta_{0,1,F}$ est stablement dominé. On peut retrouver
ce fait grâce à la proposition~\ref{equivstabdom} et à la remarque~\ref{remstabdom} ; 
nous allons les utiliser pour vérifier plus généralement
que $\eta_{\lambda,r,F}$ est
stablement dominé pour tout $\lambda\in F$ et tout $r\in \abs F$. Si $r=0$
alors $\eta_{\lambda,r,F}$ est le type simple induit
par le point $\lambda$ de $\mathbb{A}^1_F(F)=F$, et il est donc stablement
dominé. Supposons $r>0$. Le point $\eta_{\lambda,r,F}$
est alors le type associé au point du spectre valuatif
de $\mathbb{A}^1_F$ défini par la valuation~$\abs . : F(u)\to \abs F$ qui envoie
$\sum a_i(u-\lambda)^i$ sur $\max |a_i|\cdot r^i$. Comme
$\abs{F(u)}=\abs F$, le type  $\eta_{\lambda,r,F}$ est 
stablement dominé. 

Nous allons {\em a contrario} vérifier que si $r>0$ alors
$\eta_{\lambda,r^-,F}$ n'est pas stablement dominé. En effet,
il est associé au point du spectre valuatif
de $\mathbb{A}^1_F$ défini par la 
valuation~$\abs . : F(u)\to \abs F$ qui envoie
$\sum a_i(u-\lambda)^i$ sur $\max |a_i|\cdot \omega^ir^i$
(avec les notations de l'exemple (3) de \ref{typegeneral}).
On a donc $\abs{F(u)}=\omega^{\mathbb Z}\cdot \abs F\supsetneq \abs F$ ; 
par conséquent, le type  $\eta_{\lambda,r^-,F}$ est 
stablement dominé. Il n'est pas difficile, dans ce dernier cas, 
d'exhiber une transformation naturelle définissable
$f: \mathbb{A}^1_F\to \Gamma_0$ telle que
$f(\eta_{\lambda, r^-,F})$ ne soit pas simple. Il suffit de remarquer
que $\abs{u-\lambda}=\omega$, et de traduire ce fait. On note $f$ la 
transformation $F$-définissable de $\mathbb{A}^1_F\to \Gamma_{0,F}$
donnée par la formule $x\mapsto \abs{x-\lambda}$. Sa valeur en $\eta_{\lambda, r^-,F}$
est alors par la remarque qui précède
le type sur $\Gamma_{0,F}$ induit par $\omega$, qui n'est pas simple : 
ce type n'est autre que $1_{\abs F}^-$, qui est caractérisé par le fait 
qu'il est situé sur $]\epsilon ; 1[_{\abs F}$ pour tout $\epsilon <1$ dans
$\abs F$.

\section{Espaces chapeautés et énoncé du théorème principal}

{\em Références :} la référence principale est bien entendu l'article \cite{hrushovski-loeser2010}
lui-même, mais le lecteur pourra consulter avec profit les transparents d'exposés de Hrushovski 
(\cite{notes-hrush2011}). 

\medskip
On travaille dans la théorie ACVF. On 
fixe {\em pour toute la suite du texte} un corps valué $(k,G_0)$. On se donne
une clôture algébrique $k^a$ de $k$, et l'on note simplement $\mathsf M$ la catégorie
$\mathsf M_{k^a,G_0^{\mathbb Q}}$. Si $A$ est une sous-structure de
$(k^a,G_0^{\mathbb Q})$, un {\em foncteur $A$-définissable} désignera ici la restriction 
à $\mathsf M$ d'un foncteur $A$-définissable au sens précédemment utilisé ; on note
encore 
$\overline B$ et $\Gamma_0$ les restrictions respectives à $\mathsf M$ des foncteurs 
$\overline B$ (qui envoie $F$ sur $\{b(\lambda,r)_F\}_{\lambda\in F,r\in \abs F}$) et $\Gamma_0$
(qui envoie $F$ sur $\abs F$). De même, si $X$ est une $k$-variété algébrique, on 
note encore $X$ le foncteur qu'elle induit sur $\mathsf M$. 

\subsection{Les espaces chapeautés : définition, exemples de base, premières propriétés}

Soit $D$ un foncteur $(k,G_0)$-définissable, soit $F\in \mathsf M$ et soit $F'\in \mathsf M_F$. 
La remarque~\ref{stabdomfonct} assure que l'ensemble des types stablement dominés sur $D_F$ se
plonge, {\em via} l'opération d'extension canonique au modèle $F'$, dans celui des types stablement dominés
sur $D_{F'}$. Ainsi, la formation de l'ensemble des types stablement dominés sur $D_F$ est
{\em fonctorielle en $F$.}

\begin{defi}\label{defchapeau} Soit $D$ un foncteur $(k,G_0)$-définissable. On note $\widehat D$
le foncteur qui associe à un modèle $F\in \mathsf M$ l'ensemble des types stablement dominés
sur $D_F$.

\end{defi} 

La remarque~\ref{stabdomfonct} assure que la formation de $\widehat D$ est fonctorielle en $D$, 
pour les transformations naturelles $(k,G_0)$-définissables. 

\medskip
{\em Commentaires et premiers exemples.} Si $F\in \mathsf M$, tout point de $D(F)$
définit un type simple, et partant stablement dominé, 
sur $D_F$. On dispose par ce biais d'un plongement
naturel $D\hookrightarrow \widehat D$.

Supposons que $D$ «vive dans la sorte $\Gamma_0$», c'est-à-dire 
soit de la forme $\Delta\circ \Gamma_0$, où $\Delta$ est un foncteur $G_0$-définissable
dans la théorie DOAG ; c'est par exemple le cas dès que $D$ est 
un sous-foncteur de $\Gamma_0^n$. Dans ce cas, pour tout $F\in \mathsf M$, 
les types stablement dominés sur $D_F$ sont exactement les types 
simples : autrement dit, le plongement $D\hookrightarrow \widehat D$ est 
bijectif. On montre plus généralement que si $D'$ est un foncteur
$(k,G_0)$-définissable, le foncteur $\widehat {D'\times D}$ s'identifie
naturellement à $\widehat {D'}\times D$. 

Nous allons maintenant décrire $\ahat k$. Pour tout,
$F\in \mathsf M$ on note $\iota(F)$ l'application
qui envoie une boule fermée $b(\lambda,r)_F\in \overline B(F)$ 
sur son type générique $\eta_{\lambda, r,F}$. Les types de la forme
$\eta_{\lambda, r,F}$ étant stablement dominés, 
on définit ainsi un plongement 
$\iota : \overline B \hookrightarrow \ahat k$, dont
on démontre qu'il est {\em bijectif}. Par conséquent, 
$\ahat k\simeq \overline B$. Comme $\overline B$ est $(k,G_0)$-définissable, 
il en va de même de $\ahat k$. 

Ce dernier fait se généralise en partie. Avant de formuler l'énoncé correspondant, 
donnons une définition. Soit $X$ une $k$-variété algébrique, et soit $U$ un sous-foncteur
$(k,G_0)$-définissable de $X$ ; cela signifie, rappelons-le,
que $U$ est défini Zariski-localement
sur $X$ par une combinaison booléenne d'inégalités 
de la forme $\abs f\Join \lambda \abs g$, où $f$ et $g$
sont des fonctions régulières, où $\lambda\in G_0$ et
où $\Join\in\{<,>,\leq,\geq\}$. Supposons que $U$
est non vide et soit $F\in \mathsf M$. Le plus grand entier $n$
pour lequel il existe un {\em plongement} $F$-définissable de $b(0,1)_F^n$ 
dans $U_F$ ne dépend pas de $F$, et est appelé la {\em dimension} de~$U$. 

\begin{prop}\label{prodef} Soit $U$ comme ci-dessus. Le foncteur $\widehat U$
est alors $(k,G_0)$-{\em prodéfinissable} ; il est $(k,G_0)$-définissable si et seulement si
il est de dimension inférieure ou égale à $1$. 
\end{prop}

Faisons quelques commentaires. Dire qu'un foncteur est $(k,G_0)$-prodéfinissable
signifie qu'il 
est isomorphe à une limite projective de foncteurs $(k,G_0)$-définissables --
on demande de surcroît que l'ensemble d'indices ne soit 
pas trop gros, dans un sens que nous ne préciserons pas ici.
Bien entendu, pour que la limite projective en question soit canonique, il faut,
à l'instar de ce qu'on a vu dans le cas définissable, 
se donner un peu plus qu'un foncteur sur $\mathsf M$ : 
le foncteur en question doit se mettre en familles de façon évidente. 
C'est le cas en ce qui concerne~$\widehat U$
-- on dispose d'une définition naturelle de famille définissable
de types stablement dominés. 

Une bonne partie de ce qu'on a vu à propos des foncteurs
définissables s'étend au cadre des foncteurs
prodéfinissables, comme la notion de transformation 
définissable, 
celle de type, ou celle de type définissable. Si $D$ est un foncteur
$(k,G_0)$-prodéfinissable, un sous-foncteur $D'$ de $D$ sera dit : 

- {\em $(k,G_0)$-isodéfinissable} s'il existe un foncteur $(k,G_0)$-définissable
$\Delta$, et une transformation $(k,G_0)$-définissable $\Delta\to D$ qui induit
un isomorphisme $\Delta\simeq D'$ ; 

- {\em relativement $(k,G_0)$-définissable} s'il existe un foncteur $(k,G_0)$-définissable
$\Delta$, un sous-foncteur $(k,G_0)$-définissable $\Delta'$ de $\Delta$, et une 
transformation naturelle $(k,G_0)$-définissable $f: D\to \Delta$ telle que $D'=f^{-1}(\Delta')$. 

Par exemple, la transformation
naturelle $U\hookrightarrow \widehat U$ est $(k,G_0)$-définissable. Son image
(que l'on identifiera à $U$) est donc $(k,G_0)$-isodéfinissable ; on montre 
qu'elle est aussi~$(k,G_0)$-relativement définissable. Si~$V$ est 
un sous-foncteur~$(k,G_0)$-définissable de $V$ alors $\widehat V$
est relativement $(k,G_0)$-définissable dans $\widehat U$. 

\begin{rema} On montre plus précisément que le foncteur $\widehat U$ est 
{\em strictement}~$(k,G_0)$-prodéfinissable : cela signifie que pour toute transformation $(k,G_0)$-définissable
$f$ de $\widehat U$ vers un foncteur~$(k,G_0)$-définissable $D$, le foncteur image
$f(\widehat U)$ est $(k,G_0)$-définissable. 
\end{rema}

\medskip
Disons maintenant quelques mots des preuves. 

$\bullet$ La pro-définissabilité de $\widehat U$ est établie à l'aide
d'arguments
qui n'ont rien de spécifique à ACVF et s'appliquent
à bien d'autres théories. Ils
reposent pour l'essentiel sur
une propriété très générale
d'uniformité, qui dans 
le cas qui nous préoccupe dit
en gros la chose suivante : 
si, avec les notations de la proposition~\ref{deftypedef}, on fixe~$D'$ et~$\Delta$
et fait parcourir à~$t$ 
l'ensemble des types définissables et orthogonaux à $\Gamma_0$ sur $D$,
 alors le sous-ensemble
définissable de $D'(M)$ 
fourni par~{\em loc. cit.}
évolue au sein
d'une famille
$M$-définissable. 

$\bullet$ La {\em stricte} prodéfinissabilité de $\widehat U$
repose, 
{\em via} le fait qu'un type stablement dominé est contrôlé
par sa trace sur la sorte «corps résiduel», sur une propriété
de la théorie des corps algébriquement clos  : si $X$ 
est un foncteur définissable dans cette théorie, le foncteur
des types
(définissables) sur $X$ est {\em ind-définissable}. Pour se convaincre
de ce dernier point, rappelons que
si $X$ est un schéma, les types sur $X$ correspondent bijectivement
aux points du schéma $X$, donc aux fermés irréductibles de $X$
-- d'où un foncteur ind-définissable : on se ramène
au cas affine, on fixe un «degré» et un nombre
d'équations,  on fait varier
les coefficients de celles-ci, on ne garde que celles
qui décrivent un fermé irréductible, 
et l'on quotiente par la relation d'équivalence
qui identifie deux systèmes d'équations définissant le même fermé. 

$\bullet$ La définissabilité de $\widehat U$ en dimension $\leq 1$
est quant à elle une conséquence du résultat de finitude suivant, qui
lui-même résulte du théorème de Riemann-Roch : si $X$ est une courbe 
projective, irréductible et lisse de genre $g$ sur un corps algébriquement clos $F$,
le corps $F(X)$ est engendré multiplicativement par les fonctions
ayant au plus~$g+1$
pôles (avec multiplicités) ; cela généralise le fait 
que les polynômes de degré $1$ en $u$
engendrent multiplicativement $F(u)$. 

{\em À propos des fibres d'applications chapeautées.}  Soient $D$ et $D'$ deux foncteurs
$(k,G_0)$-définissables,
soit $f: D\to D'$ une transformation naturelle 
$(k,G_0)$-définissable et soit $\widehat f : \widehat D\to \widehat{D'}$ la
transformation induite. 
Soit $F\in \mathsf M$. Si $x\in \widehat{D'}(F)$, 
la fibre $\widehat f _F^{-1}(x)$ est un sous-foncteur bien défini de $\widehat D_F$. 

Supposons que $x\in D'(F)$, c'est-à-dire encore que $x$
est un type simple. Il n'est alors pas difficile de voir que
$\widehat f_F^{-1}(x)$ s'identifie naturellement à 
$\widehat{f_F^{-1}(x)}$. 

Par contre, on prendra garde que si $x$ n'est pas simple, 
le foncteur $\widehat f_F^{-1}(x)$ {\em ne s'interprète pas en général
comme un espace chapeauté.} Ce phénomène qui 
peut sembler 
un peu désagréable 
-- les fibres ne sont pas toutes des objets de la théorie -- n'est pas lié
à une lacune dans les définitions, mais à une «vraie pathologie»~de la
théorie des valuations, dont nous allons décrire une manifestation. Il est possible
d'exhiber deux modèles $F\subseteq F'$ de ACVF et un couple $(x,y)\in (F')^2$
tels que les
propriétés suivantes soient satisfaites : 

- le type sur ${\mathbb A}^1_F$ induit par $x$ est stablement dominé ; 

- le type sur ${\mathbb A}^2_F$ induit par $(x,y)$ est stablement dominé ; 

- le type sur ${\mathbb A}^1_{F(x)^a}$ induit par $y$ {\em n'est pas} $F(x)^a$-définissable,
et {\em a fortiori} pas stablement dominé 
(on désigne par $F(x)^a$ la fermeture algébrique de $F(x)$ dans $F'$). 

\medskip
Cela dit, ce défaut n'est pas rédhibitoire. Nous verrons qu'il n'interdit
pas de raisonner par fibrations ; simplement, il contraint à 
un certain nombre
de contorsions techniques lorsqu'on veut le faire, puisqu'on ne 
sait travailler que sur les fibres en les points simples. 

\subsection{La topologie sur $\widehat U$}

Pour tout $n$, on munit le foncteur $(k,G_0)$-définissable
$\Gamma_0^n$ de la topologie définissable pour laquelle 
les ouverts $F$-définissables de $\Gamma_{0,F}^n$ sont, pour
tout $F\in \mathsf M$, les sous-foncteurs de~$\Gamma_{0,F}^n$
qui peuvent être définis par une combinaison booléenne positive
d'inégalités de la 
forme~$a\prod x_i^{n_i}<b\prod x_i^{m_i}$ où $a$ et $b$ appartiennent à 
$\abs F$. Cette topologie possède les propriétés intuitives attendues : 
si $D$ est un sous-foncteur~$(k,G_0)$-définissable de~$\Gamma_0^n$, 
il est définissablement compact si et seulement si il est borné et
peut être décrit par une combinaison booléenne positive
d'inégalités {\em larges}
entre monômes.
Cette topologie en induit une sur tout segment généralisé, et plus généralement
sur tout graphe fini (pour une définition,
voir la fin de~\ref{Ensfonct}). Un graphe fini est définissablement 
compact.

La notion de topologie définissable sur un foncteur
définissable s'étend au cas des foncteurs prodéfinissables 
-- sur un modèle donné, une telle topologie est donnée par une collection 
de sous-foncteurs {\em relativement définissables.} La compacité définissable
se définit de façon analogue dans ce cadre. 

On désigne toujours par $U$ un sous-foncteur $(k,G_0)$-définissable
d'un $k$-schéma de type fini $X$. Soit $F\in \mathsf M$, soit $V$
un ouvert de Zariski de $X_F$ et soit $f$ une fonction régulière sur $V$. 
On peut voir $\abs f$ comme une transformation $F$-définissable
de~$V$ vers~$\Gamma_{0,F}$ ; elle en induit une de $\widehat V$ vers
$\widehat \Gamma_{0,F}=\Gamma_{0,F}$, que l'on note encore $\abs f$. 
Soit~$D$ un ouvert~$F$-définissable de~$\Gamma_{0,F}$ ; le foncteur
$\abs f^{-1}(D)\cap \widehat U_F$ est un sous-foncteur relativement
$F$-définissable de $\widehat U_F$. On munit $\widehat U$ de la 
topologie définissable la plus grossière pour laquelle $\abs f^{-1}(D)\cap \widehat U_F$
est un ouvert relativement~$F$-définissable
de~$\widehat U_F$ pour tout~$(F,V,f,D)$ comme ci-dessus. Indiquons quelques propriétés de cette topologie. 

\begin{itemize}

\item[$\bullet$] Pour tout $F\in \mathsf M$, 
le plongement naturel~$U(F)\hookrightarrow \widehat U(F)$ est un homéomorphisme 
de $U(F)$ (muni de la topologie déduite de celle du corps valué $F$) sur son image, 
laquelle est dense. 

\item[$\bullet$] Si $X$ est géométriquement connexe alors $\widehat X$ est 
$(k^a,G_0^{\mathbb Q})$-connexe : il ne peut s'écrire comme une union disjointe de deux
ouverts $(k^a,G_0^{\mathbb Q})$-définissables non vides. Si $X$ est connexe et {\em si~$k$ est hensélien}
alors
$\widehat X$ est $(k,G)$-connexe.

\item[$\bullet$] Supposons que $X$ est séparé,
et qu'il existe une famille finie $(X_i)$ d'ouverts affines de $X$
telle que $U=\bigcup U_i$, où $U_i$ est un sous-foncteur $(k,G_0)$-définissable
de $X_i$ possédant les propriétés suivantes : 

\begin{itemize}

\item[$\alpha)$] $U_i$ est borné relativement à un plongement fixé 
de $X_i$ dans un espace affine ; 

\item[$\beta)$] $U_i$  est «naïvement fermé», c'est-à-dire 
qu'il peut être défini par une combinaison booléenne positive
d'inégalités 
de la forme $\abs f \leq \lambda \abs g$ où $f$ et $g$ sont des fonctions régulières
sur $X_i$, et où $\lambda\in G_0$. 
\end{itemize}

Sous ces hypothèses, $\widehat U$ est définissablement
compact. Notons que si $X$ est projective et possède un recouvrement
$(X_i)$ par des ouverts affines tels que $U\cap X_i$ soit
naïvement fermé pour tout $i$, 
les hypothèses ci-dessus sont satisfaites et $\widehat U$ est définissablement
compact. En particulier, $\widehat X$ est définissablement compact dès que $X$ est projective. 
\end{itemize}

\medskip
{\em La structure d'arbre de $\phat k$.} Identifions $\ahat k$ au foncteur $\overline B$ des boules fermées
(en envoyant une boule sur son type générique). 
Soit $F\in \mathsf M$ et soit $\lambda$ 
un élément de $F$.
On note $b(\lambda,.)_F$ la transformation naturelle~$\Gamma_{0,F}\to \ahat F$ qui, sur 
un modèle $F'$ donné dans $\mathsf M_F$, envoie un élément~$r$ de~$\abs {F'}$ 
sur~$b(\lambda,r)_{F'}$. Elle est $F$-définissable, et induit un homéomorphisme
entre $\Gamma_{0,F}$ et un sous-foncteur~$F$-définissable de $\ahat F$. 
Soit~$\mu$ un (autre) élément de $F$ ; posons~$r=\abs{\lambda-\mu}$. 
Soit $I$ le $F$-segment généralisé
obtenu en concaténant $[0;r]_F$ et $[r; 0]_F$. 
On définit alors un homéomorphisme $F$-définissable de $I$ sur un sous-foncteur
$F$-définissable de $\ahat F$ en appliquant $b(\lambda,.)_F$ sur le premier segment, 
et $b(\mu,.)_F$ sur le second. 

Ce « segment généralisé joignant $\lambda$ à $\mu$~» est le seul. Plus généralement, 
on démontre que si $x$ et $y$ appartiennent à $\phat k(F)$, il existe un unique sous-foncteur
$F$-définissable~$D$ de~$\phat F$ possédant la propriété suivante : {\em il existe un $F$-segment
généralisé $I$ et un homéomorphisme~$F$-définissable~$I\simeq D$ envoyant l'origine de $I$ sur $x$
et l'extrémité de~$I$ sur $y$.} Donnons une description informelle du foncteur~$D$ : 
si $x$ et $y$ appartiennent à $\ahat k(F)$, on fait croître 
le rayon de la boule $x$ jusqu'à rencontrer la boule $y$,
puis on le diminue jusqu'à obtenir exactement $y$. Si $x\in \ahat k(F)$ et
si $y=\infty$, 
on fait croître le rayon de la boule $x$ jusqu'à l'infini (et on paramètre par l'inverse du rayon, 
le point correspondant à $0$ étant alors $y$). On dira que $D$ est le {\em segment
généralisé joignant $x$ à~$y$.} 

On dispose d'une notion naturelle d'{\em enveloppe convexe}
$C$ d'une famille finie de points de~$\phat k (F)$ : c'est la réunion des segments généralisés les joignant
deux à deux. Il existe un homéomorphisme $F$-définissable d'un arbre fini sur $C$.

\medskip
Terminons cette section en mentionnant que si $U$ est comme ci-dessus et si $D$ est un 
sous-foncteur $(k,G_0)$-définissable de $\Gamma_0^n$, ou un segment généralisé et plus
généralement un graphe fini, on munit $\widehat U\times D$ de la topologie définissable produit. 

\subsection{Liens avec les espaces de Berkovich}

On suppose dans cette section que $G_0\subset {\mathbb R}_+$
et que $k$ est complet. Le corps $k$
est donc un {\em corps ultramétrique complet}. On désigne
toujours par $X$ un $k$-schéma de type fini, et par $U$ un sous-foncteur
$(k,G_0)$-définissable de $X$. Les inégalités qui décrivent
$U$ définissent sans ambiguïté une partie semi-algébrique $U^{\rm an}$
de $X^{\rm an}$. 

Soit $F\in \mathsf M$ un corps {\em ultramétrique} et
soit $x\in \widehat U(F)\subset \widehat X(F)$. Le point $x$ est 
un type sur $X_F$, c'est-à-dire un couple formé d'un point
$\xi$ du schéma $X_F$ et d'une valuation sur $F(\xi)$, 
prolongeant celle de $F$. Comme $x$ est stablement dominé, 
il est orthogonal à $\Gamma_0$, ce qui veut dire 
que $\abs {F(\xi)}\subset \abs F \subset  {\mathbb R}_+$. Soit $\xi_0$ l'image
de $\xi$ sur le schéma $X$. La valuation sur $k(\xi_0)$ induite
par celle de $F(\xi)$ est à valeurs réelles ; elle définit donc un point 
$\theta(x)$ de $X^{\rm an}$ supporté par $\xi_0$. L'appartenance de $x$ à $\widehat U(K)$
implique immédiatement que $\theta(x)\in U^{\rm an}$. 

Avant d'énoncer les propriétés fondamentales de l'application $\theta$, 
rappelons qu'un corps ultramétrique  $F$ est dit {\em maximalement complet} 
si toute famille décroissante de boules fermées de $F$ a une intersection non vide. 
Cela équivaut à demander que $F$ n'admette pas d'extension immédiate non triviale
(une extension valuée de $F$ est dite immédiate si elle a même corps résiduel
et même groupe des valeurs que $F$). Tout corps ultramétrique $F$ 
admet une extension 
maximalement complète $F'$ qui est algébriquement close, vérifie
l'égalité $\abs {F'}={\mathbb R}_+$, et a pour corps résiduel une clôture algébrique de 
$\tilde F$ ; une telle extension est essentiellement unique. Notons qu'un corps
maximalement complet est complet (considérer une famille décroissante
de boules dont le rayon tend vers $0$).

\begin{prop} Soit $F\in \mathsf M$ un corps 
ultramétrique. 

1) L'application $\theta : \widehat U(F)\to U^{\rm an}$ est continue ; si $F=k$, c'est un homéomorphisme
sur son image. 

2) Supposons que $F$ soit maximalement complet et que $\abs F={\mathbb R}_+$. 
L'application $\theta$ est alors
une surjection 
topologiquement propre, et un homéomorphisme si $F=k$.
\end{prop}

Faisons quelques commentaires. L'assertion 1) résulte immédiatement de la définition des
topologies sur $\widehat U(F)$ et sur $U^{\rm an}$. Disons à titre d'illustration
quelques mots de l'image de $\theta$
lorsque $F=k$ (ce qui implique
que $k$ est algébriquement clos,
que sa valuation n'est pas triviale, et que~$|k|=G_0$)
et lorsque $U=X={\mathbb A}^1_k$. L'application $\theta$ envoie 
par construction 
le point $\eta_{\lambda, r,k}\in \ahat k(k)$
sur la semi-norme $\sum a_i (u-\lambda)^i\mapsto \max \abs {a_i}r^i$. On voit donc
que son image consiste exactement en l'ensemble des points {\em de type $1$ ou $2$}
de~${\mathbb A}^{1,\rm an}_k$. Ce fait s'étend à toute courbe algébrique $X$ : 
l'image de~$\theta : \widehat X(k)\to  X^{\rm an}$ est l'ensemble de points de type $1$ ou $2$
de $X^{\rm an}$. Du point de vue de la définissabilité,
qui préside à l'ensemble de la construction de Hrushovski et Loeser, cela n'a rien 
de choquant : l'expérience montre que, tant qu'on s'intéresse à des problèmes de nature
algébrique ne faisant pas intervenir d'autres nombres réels que ceux de~$\abs k$, les points
de type 1 ou 2 sont précisément les seuls points en lesquels
~«il peut se passer quelque chose». Les autres 
servent à garantir de bonnes propriétés topologiques (compacité, connexité par arcs), mais ne détectent rien
d'intéressant. Se limiter aux points de type 1 ou 2 ne fait donc perdre pour l'essentiel aucune information 
-- simplement, il faut compenser les désagréments topologiques inhérents à 
cette approche
en remplaçant les propriétés classiques par leurs
avatars modèles-théoriques : compacité définissable, existence de segments généralisés définissables joignant
deux points, etc.

En ce qui concerne l'assertion 2), sa preuve repose essentiellement sur le fait suivant 
(qui garantit la surjectivité, la propreté découlant ensuite du théorème de Tychonoff 
convenablement utilisé). Soit $F$ un corps
maximalement complet, algébriquement clos
et tel que $\abs F={\mathbb R}_+$, soit $X$ un $F$-schéma 
de type fini et soit $x$ un type sur $X$. Il correspond à un 
couple $(\xi,\abs.)$ où $\xi$ est un point de $X$ et $\abs.$
une valuation 
sur $F(\xi)$. Le type $x$ est alors stablement dominé si et seulement
si $\abs{F(\xi)}=\mathbb R_+$. {\em Notons la différence absolument cruciale
avec la proposition~\ref
{equivstabdom} : nous n'avons pas supposé {\em a priori} que $x$ est $F$-définissable. }

\medskip
Cette application $\theta$ peut être utilisée pour 
transférer certaines applications chapeautées en géométrie de Berkovich, comme le montre
le théorème ci-dessous.
\begin{theo}\label{transferberk} Soit $U$ comme ci-dessus, et soit $D$ un foncteur $(k,G_0)$-définissable
vivant dans la sorte $\Gamma_0$ ; on suppose plus précisément que $D$ est muni d'une topologie définissable, 
et qu'il admet un plongement topologique $(k,G_0)$-définissable dans $\Gamma_0^n $ pour un certain $n$ 
(c'est par exemple le cas si $D$ est un segment généralisé). Soit $V$ un sous-foncteur
$(k,G_0)$-définissable d'une $k$-variété algébrique et soit $f: \widehat U \times D \to 
\widehat V$
une transformation naturelle $(k,G_0)$-définissable et {\em continue}. Soit~$F$ une extension ultramétrique complète
de $k$. On suppose que $F$ est algébriquement clos, maximalement complet, et que $\abs F={\mathbb R}_+$. Il existe alors une unique application continue 
$$f^{\rm an} : U^{\rm an}\times D({\mathbb R}_+)\to V^{\rm an}$$ telle que $$\diagram \widehat U(F)\times D({\mathbb R}_+)
\rto^{\;\;\;\;\;\;\;\;\;\;\;f} \dto_{(\theta,{\rm Id})}&\widehat V(F)\dto^\theta\\U^{\rm an}\times D({\mathbb R}_+)\rto^{\;\;\;\;\;\;\;f^{\rm an}}&V^{\rm an}\enddiagram$$ commute. 

\end{theo} 

\begin{rema}\label{preuvetransf} L'unicité et la continuité de 
$f^{\rm an}$
proviennent immédiatement
du caractère surjectif et propre de $\theta$. L'existence demande
un peu plus de travail : le diagramme commutatif indique comment la construire, 
mais il y a des choix d'antécédents à faire et il faut vérifier que le résultat n'en dépend pas. 
La définissabilité {\em et} la continuité de $f$ sont utilisées pour cette étape. 
\end{rema}

\subsection{Énoncé du théorème principal}

Commençons par quelques conventions. Soit $D$ un sous-foncteur~$(k,G_0)$-définissable de $\Gamma_0^n$. 
On fixe un modèle $F\in \mathsf M$ ; le plus grand entier $m$ tel qu'il existe un plongement définissable 
$\prod_{1\leq i\leq n}[a_i;b_i]_F\hookrightarrow D_F$, où les $a_i$ et les $b_i$ appartiennent à $\abs F$ et où $a_i<b_i$ pour tout
$i$, ne dépend pas de $F$ ; on l'appelle la {\em dimension} de $D$ (on la prend égale à $-\infty$ si $D$ est vide). 

Si $U$ est un foncteur $(k,G_0)$-définissable et $f: U\to D$ une transformation naturelle $(k,G_0)$-définissable, nous nous
permettrons souvent de noter encore $f$ la transformation naturelle $\widehat U\to \widehat D\simeq D$ induite par $f$. 

Nous allons maintenant énoncer le théorème principal de Hrushovski et Loeser, qui porte
sur la géométrie chapeautée ; les résultats annoncés de modération des espaces de Berkovich en découlent grâce
au théorème~\ref{transferberk} ci-dessus.

\begin{theo}\label{theoprinc} Soit $(k,G_0)$ un corps valué, soit $X$ une $k$-variété algébrique {\em quasi-projective}
et soit $U$ un sous-foncteur $(k,G_0)$-définissable de $X$. Soit $\mathscr F$ une famille finie de transformations 
naturelles $(k,G_0)$-définissables de $U$ vers $\Gamma_0$, et soit $\mathsf G$ un groupe fini agissant sur $X$ (par automorphismes
de $k$-schéma) et stabilisant $U$. Il existe : 

$\bullet$ un segment généralisé $(k,G_0)$-définissable $I$, d'origine $o$ et d'extrémité $e$ ; 

$\bullet$ une transformation naturelle $(k,G_0)$-définissable et {\em continue} $h: I\times \widehat U \to \widehat U$ ; 

$\bullet$ un sous-foncteur $(k,G_0)$-isodéfinissable $S$ de $\widehat U$ ; 

$\bullet$ un sous-foncteur $(k,G_0)$-définissable $P$ de $\Gamma_0^n$ (pour un certain $n$), de dimension majorée
par celle de $X$ ;  

$\bullet$ une extension finie $L$ de $k$ contenue dans $k^a$ et un homéomorphisme $(L,G_0)$-définissable $P\simeq S$
(nous résumerons l'existence de~$L$, de~$P$ et de l'homéomorphisme évoqué en disant simplement que~$S$ est un {\em polytope}),\par

\smallskip

\noindent tels que les propriétés suivantes soient satisfaites pour tout modèle 
\mbox{$F\in \mathsf M$,} tout \mbox{$x\in \widehat U(F)$} et
tout $t\in I(F)$. 

1) On a $h(e,x)\!\in\!S(F)$, 
et $h(t,x)\!=\!x$ dès que $x\in S(F)$. On a aussi les égalités~$h(e,h(t,x))\!=\!h(e,x)$ et~$h(o,x)=x$. Nous dirons
plus brièvement que $h$ est~{\em une homotopie d'image~$S$}.   

2) Si $f\in \mathscr F$ alors $f(h(t,x))=f(x)$. 

3) Pour tout $g\in \mathsf G$ on a $g(h(t,x))=h(t, g(x))$. 

\end{theo}

\begin{rema} Ce n'est pas seulement par souci de généralité que Hrushovski et Loeser ont cherché à imposer à leur homotopie $h$ de préserver les fonctions appartenant à $\mathscr F$ (propriété 2) et d'être équivariante (propriété 3) : même pour construire $h$ sans ces contraintes, leur stratégie requiert, au cours d'un raisonnement par récurrence, de savoir construire en dimension inférieure des homotopies auxiliaires qui satisfont ce type de conditions. 
\end{rema}

\begin{rema}\label{gallois} Il n'existe pas, en général,
d'homéomorphisme $(k,G_0)$-définissable entre $S$ et un sous-foncteur 
 $(k,G_0)$-définissable de $\Gamma_0^n$. La raison est que le type d'homotopie de $\widehat X$
 n'a aucune raison d'être invariant sous l'action de Galois. 
 Par exemple, le groupe de Galois
 peut permuter les composantes connexes de $\widehat X$, si $X$ est connexe mais pas géométriquement
 connexe. Il peut aussi agir de façon plus subtile sur la topologie de $\widehat X$. 
 Considérons
 par exemple le cas où $X$ est la courbe elliptique sur~${\mathbb Q}_3$ d'équation affine
 $y^2=x(x-1)(x-3)$. On peut alors construire une homotopie comme dans le théorème, 
 dont l'image $S$ est homéomorphe à un cercle
  (ou plus précisément à son avatar modèle-théorique dans DOAG). 
 Cet homéomorphisme est~${\mathbb Q}_3(i)$-définissable, mais n'est pas ${\mathbb Q}_3$-définissable : 
 la conjugaison agit en fixant deux points sur le cercle et en échangeant les deux demi-cercles correspondants. 
 
 En termes d'espaces de Berkovich, $X_{{\mathbb Q}_3(i)}^{\rm an}$ a le type d'homotopie d'un cercle
 (c'est une courbe de Tate déployée), mais $X^{\rm an}$ a celui du quotient d'un cercle par l'action
 précédemment décrite, c'est-dire d'un segment.  
 
\end{rema}

\section{Esquisse de la preuve}

\subsection{Rétractions par déformation de $\phat k$}

Soit $U$ le sous-foncteur $k$-définissable de
~${\mathbb P}^1_k$ défini par la condition~$\abs u\leq 1$, où $u$~est la 
fonction coordonnée ; soit $V$ le sous-foncteur $k$-définissable de
$U$ défini par la condition $\abs u<1$, et soit $W$ le foncteur
${\mathbb P}^1_k\setminus U$ ; la fonction $1/u$ induit un isomorphisme
$k$-définissable $\psi : W\simeq V$. 

Soit $F\in \mathsf M$. Soit $\lambda\in F\zero$ et soit 
$r\in \abs F\zero$. 
Pour tout~$t$ appartenant
à $\abs {F\zero}$, 
on pose ~$h(t,\eta_{\lambda, r,F})=\eta_{\lambda, \max (r,t), F}$. 
Notons que~$h(0,x)=x$ 
et~$h(1,x)=\eta_{0,1,F}$ pour tout~$x\in \widehat U(F)$ ; notons
aussi que~$h(t,x)$ appartient à $\widehat V(F)$ dès que~$x\in \widehat V(F)$ et que $t<1$. 
Si~$x$ appartient à $W(F)$ on pose~ $h(t,x)=\widehat \psi^{-1}(h(t,\widehat \psi(x)))$ si~$t<1$, 
et~$h(1,x)=\eta_{0,1,F}$. On vérifie aisément que l'on a ainsi construit une
homotopie $k$-définissable $h : [0\;;1]\times \phat k\to \phat k$, dont 
l'image est le (foncteur) singleton~$\{\eta_{0,1}\}$.

On fixe
un diviseur $D$ sur ${\mathbb P}^1_k$, sans multiplicités ; 
on le voit comme un sous-ensemble
fini et Galois-invariant de ${\mathbb P}^1(k^a)$. 
Soit $C_D$ l'enveloppe convexe de $D\cup\{\eta_{0,1}\}$. 
C'est un sous-foncteur $k$-définissable de $\phat k$ qui est un polytope. Notons
que l'action de Galois sur $C_D$ n'est pas forcément triviale 
(elle l'est si $D$ est constitu\'e de~$k$-points). On définit comme suit
une homotopie $k$-définissable~$h_D :  [0\;;1]\times \phat k\to \phat k$ d'image~$C_D$. Soit
$F\in \mathsf M$ et soit $x\in \phat k(F)$. On note $\tau(x)$ le plus petit élément $t$
de~$\abs {F\zero}$ tel que~$h(t,x)\in C_D(F)$ (remarquons que~$h(1,x)=\eta_{0,1,F}\in C_D(F)$). 
Soit~$t\in \abs{F\zero}$. On pose 
$$h_D(t,x)=h(t,x)\; \;{\rm si}\;\;t\leq \tau(x)\;\;{\rm et}\;\; h_D(t,x)=h(\tau(x),x)\;\;{\rm si}\;\;\tau(x)\leq t\leq 1.$$

Remarquons que $h_D(t,x)=x$ si $x\in D$ (car $D\subset C_D(F)$), et que $h_D(t,h_D(t',x))$ est
égal à~$h_D(t',x)$ si 
$t'\leq x$, et à $h_D(t,x)$ sinon.

\subsection{Relèvement à une courbe}\label{releve}

Soit $X$ une $k$-courbe algébrique projective
et soit $\phi: X\to {\mathbb P}^1_k$ un morphisme fini. 
Le but de ce qui suit est d'exhiber un diviseur $D$
sur ${\mathbb P}^1_k$ tel que l'homotopie~$h_D$
construite ci-dessus se relève de manière unique en une homotopie
de~$[0;1]\times \widehat X$ vers~$\widehat X$ 
d'image un polytope  de dimension $\leq 1$. 
Pour ce faire, il va être nécessaire de disposer 
d'un certain contrôle sur le cardinal des
fibres de~$\widehat \phi$. La condition cruciale à ce propos s'exprime plus naturellement
dans le cadre des courbes affines. On suppose donc (pour un moment) que~$X$ est 
une courbe {\em affine}, munie d'un morphisme fini~$\phi : X\to {\mathbb A}^1_k$. 

Soit $F\in \mathsf M$, et soit $x\in \ahat k(F)$. Écrivons $x=\eta_{\lambda, r,F}$ 
pour un certain $\lambda\in F$ et un certain~$r$ appartenant à $\abs F$ (remarquons que~$r$ 
est unique, contrairement
à~$\lambda$ : c'est le rayon de la boule dont~$x$ est le type générique). 
L'application $\mathsf n$ qui à un élément~$t$ de~$\abs F$ associe le cardinal 
de~$\widehat \phi_F^{-1}(\eta_{\lambda, t, F})$ est constante par morceaux, 
par définissabilité de~$\widehat X$ et $\ahat k$. 
On dit
~que $x$ est un point de {\em ramification extérieure} de $\widehat \phi_F$ si 
$r>0$ et si $\mathsf n(t)>\mathsf n(r)$ pour $t$ suffisamment proche supérieurement de $r$
(cela ne dépend pas du choix de $\lambda$). 

\begin{lemm}\label{decroigen} Le sous-foncteur $k$-définissable de $\ahat k$
qui associe à $F\in \mathsf M$
l'ensemble des points de ramification extérieure de $\widehat \phi_F$ est fini.
\end{lemm}

Nous allons nous contenter de donner quelques indications de la preuve. 
L'idée, grossièrement exprimée, est de prouver que $\mathsf n$ a «~génériquement~»~tendance
à croître lorsque le rayon diminue ; des arguments de définissabilité 
spécifiques à la droite affine (et notamment la caractérisation 
des sous-foncteurs définissables de cette dernière comme des «~fromages suisses~»)
permettent ensuite 
de conclure
à la finitude du lieu des points de ramification extérieure de $\widehat \phi$ --  qui sont
en quelque sorte les points en lesquels on observe un comportement
non générique. 

Expliquons un peu plus avant ce que nous entendons lorsque nous disons
que $\mathsf n$~a génériquement tendance à croître lorsque le rayon diminue. 
Fixons $F\in \mathsf M$ et un point~$\eta_{\lambda, r,F}\in \ahat k(F)$ avec $r>0$. 
Soit $s$ un élément de $\abs F$ strictement inférieur à~$r$. Choisissons un 
modèle~$F'\supseteq F$ et une réalisation~$a\in F'$ 
du type~$\eta_{\lambda, r,F}$. 
Choisissons alors un modèle~$F''\supseteq F$ et une réalisation~$b\in F''$ du type
$\eta_{a,s,F'}$. L'élément $b$ est alors lui-même une réalisation
de  $\eta_{\lambda, r,F}$ (on a en effet~$|b-\mu|=r$ pour tout
élément~$\mu$ de~$F$ tel que~$|\lambda-\mu|\leq r$). 

Le nombre $N$ d'antécédents de $\eta_{\lambda, r, F}$ sur $\widehat X(F)$ est alors
égal au cardinal de l'ensemble~$E$ des types 
sur $X_{F(b)}$  admettant une réalisation $z$
sur $F''$ telle que $\phi(z)=b$ (en effet, si deux tels antécédents
sont discernables par une formule à paramètres
dans~$F(b)$, ils le sont par une formule à paramètres
dans~$F$ : il suffit de remplacer partout~$b$ par~$\phi(x)$ où~$x$ est la variable codant
un élément de $X$) ; pour la même raison,
le nombre $N'$ d'antécédents de 
$\eta_{a,s,F'}$ sur $\widehat X(F)$ est
égal au cardinal de l'ensemble $E'$ des types 
sur $X_{F'(b)}$  admettant une réalisation $z$
sur $F''$ telle que $\phi(z)=b$. On dispose d'une surjection 
naturelle de $E'$ vers $E$ (si deux antécédents de $b$ par 
$\phi$ ne peuvent être discernés par une formule à paramètres
dans $F'(b)$, ils ne peuvent {\em a fortiori} l'être par une formule
à paramètres dans $F(b)$). Par conséquent, $N'\geq N$. Autrement dit, 
on a établi l'assertion (informelle)
suivante, qui est celle que nous avions en vue : {\em le type générique d'une boule $F$-générique de rayon $s<r$ 
contenue dans $b(\lambda, r)$ a plus d'antécédents sur $\widehat X$ que le type générique 
de $b(\lambda,r)$.} 

\medskip
Revenons au problème initial, à savoir celui du relèvement des homotopies. 
On désigne donc à nouveau par $X$
une $k$-courbe algébrique projective 
munie d'un morphisme
fini $\phi: X\to {\mathbb P}^1_k$. Le choix d'une coordonnée 
$u$ sur ${\mathbb P}^1_k$ fournit deux cartes affines. Soit $D$ un diviseur 
sur ${\mathbb P}^1_k$ poss\'edant les propriétés suivantes : 

$\bullet$ la restriction de $\phi$ \`a $X\setminus\phi^{-1}(D)$ admet une factorisation $X\setminus \phi^{-1}(D)\to X'\to {\mathbb P}^1_k$ o\`u $X\setminus \phi^{-1}(D)\to X'$ 
est radiciel et o\`u $X'\to {\mathbb P}^1_k$ est \'etale ; 

$\bullet$ le polytope $C_D$ contient les points de ramification extérieure de 
$\widehat \phi$ au-dessus de chacune des deux cartes affines
de ${\mathbb P}^1_k$. 

Notons qu'un tel diviseur existe toujours en vertu du lemme~\ref{decroigen}. Nous allons alors
expliquer brièvement comment montrer qu'il existe une unique homotopie
$h'_{\phi,D}$ de~$[0;1]\times \widehat X$ vers $\widehat X$ relevant $h_D$. Fixons
$F\in \mathsf M$. 
Soit~$y\in \phat k (F)\setminus C_D(F)$
et soit~$x$ un antécédent de~$y$ sur~$\widehat X(F)$. 
Soit $\tau(y)$ le plus grand $t\in \abs{F\zero}$ tel que $h_D(\tau(y),y)=y$ ;
on a~$\tau(y)<1$ car~$y\notin C_D$. Il existe dès lors
un élément
$t_0\in \abs {F\zero}$ qui est strictement supérieur à~$\tau(y)$, 
et un voisinage~$\Omega$ de $x$ dans~$\widehat X(F)$ 
sur lequel tout point de la forme $h(t,y)$ avec $0<\tau(y)<t_0$ a
{\em au moins} un antécédent sur~$\Omega$ (on montre
en effet que~$\widehat \phi_F$ est ouvert). On peut 
de plus choisir $\Omega$ et $t$ de sorte que tout point de 
la forme $h(t,y)$ avec $0<\tau(y)<t_0$ ait
{\em exactement} un antécédent sur $\Omega$. En effet, si $\tau(y)>0$ cela provient du fait 
que $y$ n'est pas un point de ramification extérieure de $\widehat \phi_F$ (car tous ces points
appartiennent à $C_D(F)$), ce qui force le nombre d'antécédents de $h(t,y)$ à être (au plus)
égal au nombre d'antécédents
de $y$ lorsque $t$ est strictement supérieur à $\tau(y)$ et suffisamment proche de celui-ci. 
Et si $\tau(y)=0$ le point $y$ appartient à ${\mathbb P}^1_k(F)\setminus C_D(F)$ ; le point
$x$ appartient alors à $X(F)$. 
Par choix de $D$, la fl\`eche $X\to {\mathbb P}^1_k$ se d\'evisse, 
au voisinage de~$x$, en une fl\`eche radicielle (qui induit un hom\'eomorphisme entre les espaces chapeaut\'es sous-jacents) suivie d'une fl\`eche \'etale~; 
un avatar du théorème des fonctions implicites
assure alors que $\phi$ induit un homéomorphisme entre un voisinage de $x$ dans $\widehat X(F)$ et
un voisinage de $y$ dans $\phat k(F)$, d'où l'assertion. 

On peut alors définir $h'_{\phi,D}(x,t)$ pour tout $x\in X(F)$ et tout $t\in \abs{F\zero}$ : si $x$
appartient à~$\phi^{-1}(D)$, 
on pose~$h'_{\phi,D}(t,x)=x$ pour tout~$t$. Sinon, soit~$\tau$ le plus petit élément~$t$
de~$\abs{F\zero}$ tel que~$h(t,\phi(x))\in C_D(F)$. Nous avons établi ci-dessus une 
propriété «~d'unique 
relèvement de $h(.,x)$
dans le sens des temps croissants~». 
Celle-ci, combin\'ee \`a des arguments de compacit\'e d\'efinissables 
autorisant les «~passages \`a la limite \`a droite~»,  assure l'existence 
d'un unique rel\`evement d\'efinissable et continu du chemin
$$t\mapsto h_D(t,\phi(x)),\;\;\;0\leq t\leq \tau$$
en un chemin
$$t\mapsto h'_{\phi,D}(t,x), \;\;\;0\leq t\leq \tau$$
tel que 
$h'_{\phi,D}(0,x)=x$.   

On obtient ainsi une transformation naturelle continue et 
$k$-définissable $$h'_{\phi,D}: [0\;;1]\times X\to \widehat X.$$
En fait, $h'_{\phi,D}$ jouit d'une propriété de continuité {\em renforcée}, que nous ne détaillerons pas ici, 
et qui garantit son prolongement en une homotopie $k$-définissable
~$$h'_{\phi,D}: [0\;;1]\times \widehat X\to \widehat X.$$ 
Par construction, 
$h'_{\phi,D}$ relève $h_D$ ; son unicité découle de sa construction et de la densité de
$X$ dans $\widehat X$. Que $h'_{\phi,D}$ soit un polytope de dimension $\leq 1$ 
 résulte de la $k$-définissabilité de $\widehat X, \phat k$ et $\widehat f$, 
 du fait que $\widehat f$ 
est à fibres finies, et du fait que $C_D$ est un polytope de dimension $\leq 1$. 

\subsection{La récurrence : préliminaires généraux}

On note $n$ la dimension de $X$. Le théorème à démontrer est trivial si $n=0$ ; 
on suppose que $n>0$ et que le théorème est vrai en dimension $<n$. 

Par des méthodes élémentaires, 
on construit une variété projective équidimensionnelle
contenant $X$, de dimension $n$, et à laquelle l'action de $\mathsf G$
s'étend ; on peut donc supposer que $X$ est projective et
équidimensionnelle. On peut aussi, incluant 
dans la famille~$\mathscr F$
les valuations des fonctions utilisées dans la description de $U$, 
se ramener au cas où~\mbox{$U=X$~:} en effet, le théorème appliqué
dans le cas où~$U=X$ fournira une homotopie $I\times \widehat X\to \widehat X$ qui
préservera~$\widehat U$ (puisqu'elle préservera les fonctions\footnote{Les éléments de $\mathscr F$
sont {\em stricto sensu} des transformations naturelles ; nous nous permettrons
de les qualifier de fonctions.} appartenant à $\mathscr F$), 
donc induira une homotopie $I\times \widehat U\to \widehat U$ ayant les
propriétés voulues  (notons que son image~$S\cap \widehat U$ 
s'identifiera à un sous-foncteur $(L,G_0)$-définissable de $P$,
puisque $\widehat U$ est relativement~$(k,G_0)$-définissable dans $\widehat X$). 

Soit $f\in \mathscr F$. Sa définition fait intervenir un certain nombre de paramètres
appartenant à $G_0$. En autorisant ceux-ci à varier, on obtient une transformation
naturelle $k$-définissable $X\to \underline{{\rm Hom}}(\Gamma_0^m, \Gamma_0)$ 
pour un certain $m$, qu'il suffit de préserver le long des trajectoires de l'homotopie
pour que $f$ soit préservée. Or une telle transformation naturelle peut être identifiée
(en considérant les paramètres intervenant dans la description d'un
élément de  $\underline{{\rm Hom}}(\Gamma_0^m, \Gamma_0)$) à une transformation 
naturelle~$k$-définissable de~$X$ vers~$\Gamma_0^p$ pour un certain $p$ ; en la 
composant
avec les différentes projections, on obtient une famille finie de fonctions
$k$-définissables de $X$ vers $\Gamma_0$ dont la préservation entraîne celle de $f$. Par
ce procédé, on élimine les paramètres appartenant à $G_0$ ; autrement dit, on peut
supposer que $(k,G_0)=k$ (ou encore que~$G_0=\abs k$). La notion de définissabilité sur $k$ ne change 
pas si l'on remplace $k$
par son hensélisé, puis par sa clôture parfaite. On peut donc le supposer parfait
et hensélien. 

Toute fonction $f\in \mathscr F$
est {\em in fine}
décrite au moyen de la valuation de certains polynômes
homogènes (une fois $X$ plongé dans un espace projectif),
ce qui permet de supposer que les éléments de $\mathscr F$ sont 
des fonctions {\em continues}. 
On peut par ailleurs saturer $\mathscr F$
de façon à le rendre globalement $\mathsf G$-invariant puis, en remplaçant chaque~$\mathsf G$-orbite
 de~$\mathscr F$ par la liste des fonctions
qui la constituent {\em rangées dans l'ordre croissant}, 
que~$\mathscr F$ est constitué de fonctions 
individuellement~$\mathsf G$-invariantes. 
\subsection{La récurrence : préparation géométrique}

On peut supposer
$X$ réduite, donc génériquement lisse (le corps $k$ est parfait). 
Par des méthodes standard de géométrie algébrique,  {\em utilisant
de façon fondamentale le fait que $X$ est projective}, 
on montre l'existence d'un diagramme commutatif
$$\diagram X'\times_{{\mathbb P}^{n-1}_k}U\dto_\phi\rto&X'\ddto\rto &X\\
{\mathbb P}^1_k\times_kU\dto&&\\U\rto&{\mathbb P}^{n-1}_k&\enddiagram$$
et d'un diviseur~$D_0$ sur $X'$ tels que :

$\bullet$ la fl\`eche $X'\to X$ est un \'eclatement au-dessus d'un sous-sch\'ema 
ferm\'e de dimension nulle de $X$ ;

$\bullet$ le diviseur~$D_0$ est fini sur~${\mathbb P}^{n-1}_k$ ; 

$\bullet$ l'action de $\mathsf G$ se rel\`eve \`a $X'$, et stabilise $D_0$ ;

$\bullet$ le sch\'ema $U$ est un ouvert non vide de ${\mathbb P}^{n-1}_k$, et l'image r\'eciproque de ${\mathbb P}^{n-1}_k\setminus U$ sur $X'$ est un diviseur $\Delta$ de $X'$ ;

$\bullet$ la fl\`eche $X'\to {\mathbb P}^{n-1}_k$ est \'equivariante, et sa fibre g\'en\'erique est de dimension $1$ ;

$\bullet$ la fl\`eche $\phi$ est \'equivariante et finie ;

$\bullet$ le diviseur $D_0$ contient les diviseurs exceptionnels de l'\'eclatement $X'\to X$, ainsi que $f^{-1}(0)\cap Y$ pour toute composante irr\'eductible $Y$ de $X'$ et toute $f$ appartenant \`a l'image inverse de $\mathscr F$ sur $X'$ qui ne s'annule pas identiquement sur~$Y$ ; 

$\bullet$ il existe un morphisme étale et~$\mathsf G$-équivariant~$\pi : (X'\setminus D_0)\to {\mathbb A}^n_k$ (en particulier,
la variété~$X'\setminus D_0$ est lisse). 

%$$\diagram X'\rto \dto_{\phi} &X\dto\\E\dto \rto &{\mathbb P}^n_k\\
%{\mathbb P}^{n-1}_k&\enddiagram$$ et d'un diviseur $\mathsf G$-équivariant $D_0$ sur $X$ tels que les 
%propriétés suivantes soient satisfaites : 
%
%$\bullet$ l'action de $\mathsf G$ se relève à $X'$, et les flèches $X'\to E$ et $X\to {\mathbb P}^n_k$ sont
%$\mathsf G$-équivariantes ; 
%
%$\bullet$ la flèche $(X\setminus D_0)\to {\mathbb P}^n_k$ est étale ; en particulier, $D_0$ contient le lieu singulier de $X$ ; 
%
%$\bullet$ la flèche $E\to {\mathbb P}^n_k$ est l'éclatement en un $k$-point qui n'est pas situé sur l'image de $D_0$, et $E\to {\mathbb P}^{n-1}_k$
%est le morphisme standard ; 
%
%$\bullet$ le carré du haut est cartésien ; 
%
%$\bullet$ l'image réciproque $D_0'$ de $D_0$ sur $X'$ est finie sur ${\mathbb P}^{n-1}_k$ ; 
%
%$\bullet$ le diviseur $D_0$ contient $f^{-1}(0)\cap Y$ 
%pour toute composante irréductible $Y$ de $X$ et toute $f\in \mathscr F$
%non identiquement nulle sur $Y$. 

\medskip
Il suffit alors de montrer le théorème pour $X'$, en rajoutant à la famille $\mathscr F$, ou plus précisément à 
son image réciproque sur $X'$, un certain nombre de fonctions continues
décrivant la réunion $Z$ des 
diviseurs exceptionnels de l'éclatement 
$X'\to X$ ; 
on modifie
$\mathscr F$ comme expliqué plus haut pour qu'elle soit constituée de fonctions
individuellement $\mathsf G$-invariantes. 
L'homotopie~$h'$ construite sur $\widehat {X'}$ stabilisera alors $\widehat Z$. Chacune des
composantes connexes de $\widehat Z$ s'envoyant sur un point (simple)
de $\widehat X$, on pourra descendre $h'$ et conclure. 
On peut donc supposer que $X'=X$.

\subsection{L'homotopie fibre à fibre}

On choisit une fonction coordonnée $u$ 
sur ${\mathbb P}^1_k \times_k U$. 

Soit $F$ une extension de $k$ 
et soit $x\in U(F)$ ; on notera les fibres en $x$ des divers
objets en jeu par un $x$ en indice. Par construction, la
flèche~$\phi_x : X_x\to {\mathbb P}^1_F$ est finie.
 On sait donc, d'après ce qui a été fait au~\ref{releve}, 
que pour tout diviseur~$D_x$ sur~$ {\mathbb P}^1_F$
contenant
l'image de~$D_{0,x}$
et suffisamment gros, l'homotopie $h_{D_x}$ admet un unique relevé $h'_{\phi,D_x}$ 
à~$\widehat{X_x}$. On peut même faire en sorte que ce relevé préserve les fonctions appartenant à $f$ : 
en effet, chacune d'elle est continue à valeurs
dans~$\Gamma_{0,F}$, et
on peut montrer qu'une telle fonction 
est localement constante en dehors d'un graphe fini 
contenu dans $\widehat{X_x}$ ; il suffit
alors de prendre $D_x$ assez gros pour que l'image réciproque de son enveloppe convexe 
contienne tous les «~graphes de variation~» des fonctions appartenant à $f$. L'unicité
de $h'_{\phi,D_x}$ 
et le fait que chaque fonction $f$ soit invariante par $\mathsf G$ garantissent
l'équivariance
de $h'_{\phi,D_x}$. 

Un raisonnement fondé sur la compacité permet de mettre les diviseurs $D_x$ en famille. 
Plus précisément, il existe, 
quitte à restreindre~$U$, 
un diviseur $D$ sur ${\mathbb P}^1_k\times_k U$,
contenant l'image de~$D_0\cap (X\times_{{\mathbb P}^{n-1}_k}U)$ 
et qui possède les propriété suivantes : {\em le morphisme $D\to U$ est fini, 
et pour toute extension $F$ 
de $k$ et tout $F$-point $x$ de $U$, le diviseur $D_x$ est tel que $h_{D_x}$ admette un unique relevé
$h'_{\phi,D_x}$ à $\widehat{X_x}$, préservant les fonctions appartenant à $\mathscr F$ et commutant
à l'action de $\mathsf G$.} 

Notons $\widehat{X/{\mathbb P}^{n-1}_k}$ le sous-foncteur (relativement $k$-définissable) de $\widehat X$ qui envoie
un modèle $F$ sur l'ensemble des points de $\widehat X(F)$ situés au-dessus d'un point {\em simple} de~$\phat k(F)$. 
Soit $h_c$ la transformation naturelle $k$-définissable
$$[0\;;1]\times X\to \widehat{X/{\mathbb P}^{n-1}_k}$$ définie comme suit. Soit $F\in \mathsf M$, soit $y\in X(F)$, soit
$x$ son image sur ${\mathbb P}^{n-1}_k(F)$ et soit $t\in \abs {F\zero}$. Si $y\in \Delta(F)$, 
on pose $h_c(t,y)=y$ ; sinon, on pose 
$$h_c(t,y)=h'_{\phi, D_x}(t,y)\in \widehat {X_x}(F)\subset \widehat{X/{\mathbb P}^{n-1}_k}(F).$$

La transformation naturelle $k$-d\'efinissable $h_c$ a été définie de façon beaucoup trop brutale pour être continue. 
Elle l'est toutefois
lorsqu'on la restreint à 
\mbox{$[0\;;1]\times ((X\!\setminus\!\Delta)\cup D_0)$.} Elle jouit même sur ce domaine d'une
propriété de continuité renforcée (que nous avons déjà eu l'occasion d'évoquer à la fin 
du~\ref{releve}) qui assure qu'elle s'étend en une homotopie 
$$h_c :  [0\;;1]\times \widehat{(X\setminus \Delta)\cup D_0}\to  \widehat{(X\setminus \Delta)\cup D_0}$$ qui préserve ce qui doit l'être. 
L'image $\Upsilon$ de $h_c$ est un {\em polytope relatif} sur $\widehat{{\mathbb P}^{n-1}_k}$
de dimension relative $\leq 1$.

\begin{rema} Au-dessus de $\widehat{{\mathbb P}^{n-1}_k}-\widehat U$ on dispose par
construction d'un isomorphisme entre $\Upsilon$ et $\widehat D_0$ ; le polytope relatif 
~$\Upsilon$ est donc à fibres finies au-dessus de~$\widehat{{\mathbb P}^{n-1}_k}-\widehat U$. 
Ailleurs, ses fibres sont en général de dimension $1$.

\end{rema}

\subsection{L'homotopie de la base}

Le but est maintenant de construire une homotopie $J\times \Upsilon\to \Upsilon$
dont l'image soit un polytope $S$
de dimension $\leq n$,
qui préserve les fonctions $f\in \mathscr F$
et soit ${\mathsf G}$-équivariante. L'idée consiste à exhiber une homotopie
de $\widehat{{\mathbb P}^{n-1}_k}$ qui se relève de manière convenable 
à $\Upsilon$. Or ce problème relativement polytopal est en fait contrôlé
algébriquement par un morphisme fini, comme le montre la proposition
suivante (qui vaudrait pour tout polytope
relatif, la forme précise de~$\Upsilon$ importe peu). 

\begin{prop} Il existe
un revêtement fini quasi-galoisien $T\to {\mathbb P}^{n-1}_k$, et une famille finie
$\mathscr G$ de fonctions $k$-définissables de $T$ vers $\Gamma_0$ telle
que pour toute homotopie $$h: I\times \widehat{{\mathbb P}^{n-1}_k}\to \widehat{{\mathbb P}^{n-1}_k},$$
les assertions suivantes soient équivalentes : 

a) l'homotopie $h$ admet un relèvement à $\Upsilon$, préservant les fonctions appartenant à~$\mathscr F$
et $\mathsf G$-équivariant ; 

b) l'homotopie $h$ admet un relèvement à $\widehat T$ préservant les fonctions appartenant à~$\mathscr G$. 
\end{prop}

{\it En vertu de l'hypothèse de récurrence}, il existe une homotopie $h'$ sur $\widehat T$ préservant les fonctions 
appartenant à $\mathscr G$, commutant à l'action de $\mathsf H:={\rm Gal}(T/{\mathbb P}^{n-1}_k)$, et 
dont l'image est un polytope de dimension $\leq n-1$. On note
$J$ le segment généralisé où vit son paramètre temporel. 

Comme $(T/\mathsf H)\to {\mathbb P}^{n-1}_k$ est radiciel,
il induit un homéomorphisme entre les espaces chapeautés sous-jacents. 
L'homotopie~$h'$ descend dès lors
en 
une homotopie $$h : J\times  \widehat{{\mathbb P}^{n-1}_k}\to \widehat{{\mathbb P}^{n-1}_k}$$
dont l'image est un polytope de dimension $\leq n-1$.
Par construction, $h$ satisfait l'assertion b) de la proposition ci-dessus, 
et partant l'assertion a). 
On obtient ainsi une homotopie 
$h_b : J\times \Upsilon \to \Upsilon$. Comme $h_b$ relève $h$, 
l'image de $h_b$ est relativement polytopale \`a fibres de dimension $\leq 1$ sur un polytope de dimension $\leq n-1$, et est de ce fait un polytope de dimension $\leq n$.

\subsection{L'homotopie d'inflation}
La concaténation $h_b\diamond h_c$ de $h_b$ et $h_c$ (on commence par $h_c$, puis on applique $h_b$) 
définit une homotopie $$([0\;;1]\diamond J)\times 
 \widehat{(X\setminus \Delta)\cup D_0}\to  \widehat{(X\setminus \Delta)\cup D_0}$$ (on note aussi $\diamond$ la concaténation des segments
 généralisés), dont l'image $S$ est un polytope. 
 
 Pour obtenir une homotopie définie sur $\widehat X$ tout entier, on a recours à une homotopie dite d'{\em inflation} $h_{\rm inf} : J'\times \widehat X\to \widehat X$
 qui possède la propriété suivante : pour tout $F\in \mathsf M$, pour tout $t\in J'(F)$ et tout $x\in \widehat X(F)$, on a $h(t,x)=x$ si $x\in \widehat{D_0}(F)$ et $h(t,x)\notin \widehat{\Delta}(F)$ 
 si~$x\notin \widehat{D_0}(F)$ et si $t$ n'est pas l'origine de $J'$. On exige en outre que $h_{\rm inf}$ préserve les fonctions appartenant à $\mathscr F$ et soit $\mathsf G$-équivariante. 
 
 \medskip
 Disons quelques mots sur la façon dont cette homotopie $h_{\rm inf}$ est construite.  
Par choix du diviseur~$D_0$, il existe un morphisme étale et~$\mathsf G$-équivariant 
$\pi : X\setminus D_0\to {\mathbb A}^n_k$. Comme~$\pi$
est étale, il induit pour tout~$F\in \mathsf M$
et tout~$x\in X(F)\setminus D_0(F)$ un homéomorphisme d'un voisinage de~$x$ dans~$\widehat X(F)\setminus \widehat D_0(F)$
sur un ouvert de~$\widehat{{\mathbb A}^n_k}(F)$. On procède maintenant comme suit. 

\medskip
$\bullet$ On construit une homotopie~$h_0 : J'\times {\mathbb A}^n_k\to \widehat{{\mathbb A}^n_k}$ consistant peu ou prou, comme en dimension $1$,
à «~faire croître le rayon des boules». À l'exception de son origine, tous les points d'une trajectoire de~$h_0$
sont Zariski-génériques : si~$F\in \mathsf M$, si~$x\in {\mathbb A}^n (F)$ et si~$Z$ est un fermé de Zariski strict de~${\mathbb A}^n_F$
alors~$h_0(t,x)\notin \widehat Z(F)$ si~$t$ n'est pas l'origine de~$J'$. 

$\bullet$ La propriété topologique du morphisme~$\pi$ que nous avons mentionnée ci-dessus permet de construire
«un relevé partiel de~$h_0$ à~$X\setminus D_0$ avec temps d'arrêt», c'est-à-dire une homotopie~$h_1 : J'\times (X\setminus D_0) \to \widehat{ (X\setminus D_0)}$
possédant la propriété suivante : pour tout~$F\in \mathsf M$ et tout~$x\in X(F)\setminus D_0(F)$, il existe un «temps d'arrêt»~$\tau(x)$ différent
de l'origine de~$J'$ tel
que~$\pi(h_1(t,x))=h_0(t,\pi(x))$ pour~$t\leq \tau(x)$ et~$h_1(t,x)=h_1(\tau(x),x)$ pour~$t\geq \tau(x)$. 
À l'exception de son origine, tous les points d'une trajectoire de~$h_1$
sont Zariski-génériques (en particulier, ils ne sont pas situés sur~$\widehat \Delta$). 

\medskip
De plus on peut, en diminuant éventuellement les temps d'arrêt, faire en sorte : 

- que~$h_1$ soit~$\mathsf G$-équivariante (car~$\pi$ est~$\mathsf G$-équivariant) ; 

- qu'elle préserve les fonctions appartenant à~$\mathscr F$ (car celles-ci, étant 
par construction continues et 
\`a lieu des z\'eros ouvert en dehors de $D_0$, sont constantes au voisinage de tout point {\em simple} non situé sur $D_0$) ; 

- qu'elle se prolonge par continuité en une homotopie~$h_{\rm inf}: J'\times X\to \widehat X$, qui fixe~$\widehat D_0$ point par point en tout temps (il
suffit de faire tendre le temps d'arrêt vers l'origine de~$J'$ lorsqu'on se rapproche de~$\widehat D_0$).

\medskip
L'homotopie $h_{\rm inf}$ ainsi obtenue jouit de la propriété de continuité renforcée que nous avons
plusieurs fois évoquée, qui permet de l'étendre en une homotopie, encore notée~$h_{\rm inf}$, 
qui va de~$J'\times \widehat X$ vers~$\widehat X$ et possède les propriétés requises.

\subsection{L'homotopie polytopale}

On considère la concaténation $h_b\diamond h_c\diamond h_{\rm inf}$. C'est une transformation naturelle~$k$-définissable
 de $I\times \widehat X$ vers $\widehat X$, où $I$ est un segment généralisé, qui a presque toutes les propriétés requises, à une exception près : son image $\Sigma$ 
 est bien un polytope
 (c'est par construction une partie $k$-définissable 
 du polytope image de $h_b\diamond h_c$), mais {\em $\Sigma$ n'est plus nécessairement fixée point par point au cours du temps} : l'homotopie $h_{\rm inf}$ a semé la pagaille. Pour remédier
 à ce problème, Hrushovski et Loeser introduisent une quatrième homotopie $h_{\rm pol}$, qui est essentiellement construite dans le monde polytopal, par un procédé relativement
 technique que nous ne décrirons pas ici. La concaténation $$h_{\rm pol}\diamond h_b\diamond h_c\diamond h_{\rm inf}$$ répond alors aux conditions du théorème.

% \bibliographystyle{smfplain}
%\bibliography{aclab,ad}

\end{document}